\documentclass[a4paper,11pt]{article}
\usepackage[utf8]{inputenc}
\usepackage[T1]{fontenc}
\usepackage[french,english]{babel}
\usepackage{amssymb}
\usepackage{enumerate}
\usepackage{amsthm}
\usepackage{amsmath}
\usepackage{mathrsfs}
\usepackage{longtable}
\usepackage{caption}
\usepackage{relsize}
\usepackage{xcolor}
\usepackage{enumitem}
\usepackage[document]{ragged2e}
\usepackage{hyperref}
\usepackage[margin=3cm]{geometry}
\usepackage{bm}
\usepackage{bbm}
\usepackage[capitalize]{cleveref}
\usepackage{thmtools}
\usepackage{cases}
\usepackage{libertine}
\usepackage{changepage}
\usepackage{csquotes}

\usepackage[title]{appendix}
\hypersetup{
    colorlinks=true,
    linkcolor=[rgb]{0.2, 0, 0.7},
    filecolor=magenta,
    urlcolor=cyan,
    bookmarks=true,
    citecolor=[rgb]{0,0.4,0},
}

\let\svthefootnote\thefootnote
\newcommand\freefootnote[1]{%
  \let\thefootnote\relax%
  \footnotetext{#1}%
  \let\thefootnote\svthefootnote%
}

\crefformat{equation}{#2\normalfont{{(\textup{#1})}}#3}

\newcommand{\RR}{\mathbb{R}}

\newcommand{\PP}{\mathbb{P}}

\newcommand{\cN}{\mathcal{N}}

\newcommand{\cT}{\mathcal{T}}
\newcommand{\cL}{\mathcal{L}}
\newcommand{\cR}{\mathcal{R}}

\newcommand{\cI}{\mathcal{I}}

\newcommand{\cF}{\mathcal{F}}
\newcommand{\cK}{\mathcal{K}}
\newcommand{\cH}{\mathcal{H}}

\newcommand{\cS}{\mathcal{S}}

\newcommand{\norm}[2]{ \left \| {#1}  \right \|_{#2}}
\newcommand{\module}[1]{\left \lvert {#1} \right \rvert}
\DeclareMathOperator{\Lip}{Lip}
\DeclareMathOperator{\diam}{diam}

\DeclareMathOperator{\supp}{supp}

\DeclareMathOperator{\argmin}{argmin}

\DeclareMathOperator{\diver}{div}
\newcommand{\setcomp}{\mathsf{c}}
\DeclareMathOperator{\sgn}{sgn}

\DeclareMathOperator{\co}{co}

\renewcommand{\epsilon}{\varepsilon}
\newcommand{\tr}{\mathrm{Tr}}

\theoremstyle{plain}
\newtheorem{thm}{Theorem}[section]
\newtheorem{cor}[thm]{Corollary}
\newtheorem{lem}[thm]{Lemma}
\newtheorem{prop}[thm]{Proposition}

\theoremstyle{definition}
\newtheorem{defi}[thm]{Definition}

\theoremstyle{remark}
\newtheorem{rem}[thm]{Remark}

\crefname{thm}{Theorem}{Theorems}
\crefname{cor}{Corollary}{Corollaries}
\crefname{lem}{Lemma}{Lemmata}
\crefname{prop}{Proposition}{Propositions}
\crefname{def}{Definition}{Definitions}
\crefname{rem}{Remark}{Remarks}
\crefname{ex}{Example}{Examples}

\title{Some error estimates for semi-discrete finite element approximations of stable solutions to mean field game systems}
\author{Jules Berry \thanks{Université Paris-Saclay, CNRS, CentraleSupélec, Laboratoire des signaux et systèmes, 91190, Gif-sur-Yvette, France. Email : {\fontfamily{cmtt}\selectfont jules.berry@centralesupelec.fr}.}}
\date{\today}

\begin{document}
\maketitle
\justify
\freefootnote{The author would like to thank Olivier Ley and Francisco J. Silva for their remarks during the preparation of this work. He also wishes to thank the two anonymous referees for their valuable comments that helped to greatly improve the quality of the paper. This work was partially supported by the ANR (Agence Nationale de la Recherche) through the COSS project ANR-22-CE40-0010 and the Centre Henri Lebesgue ANR-11-LABX-0020-01.}
\begin{abstract}
We derive a priori error estimates for semi-discrete finite element approximations of stable solutions to time-dependent mean field game systems with Dirichlet boundary conditions. Expressing solutions to the MFG system as zeros of a nonlinear abstract mapping, we show that the stability of solutions is equivalent to the invertibility of its differential. This characterization allows us to apply the  Brezzi-Rappaz-Raviart approximation theorem in combination with discrete $L^p$ maximal regularity estimates to prove existence of solutions to the semi-discrete MFG system and to derive the error estimate.
\end{abstract}

{\small
\begin{adjustwidth}{0.8cm}{0.8cm}
\textbf{Mathematics subject classification.} 65M60, 35Q89, 35K55, 65M15.\\
\textbf{Keywords.} Mean field games, finite element method, semi-discrete approximation, a~priori error estimates, stable solutions, Brezzi--Rappaz--Raviart theorem.
\end{adjustwidth}
}

\section{Introduction}
Mean field games (MFG for short) form a theoretical framework for analyzing Nash equilibria of dynamic games involving a very large number of identical interacting agents, where each agent has negligible individual influence on the system. They were introduced independently by Lasry and Lions \cite{LL2006,LL2006a,LL2007} and Huang, Caines and Malhamé \cite{HCM2006,HCM2007}. In the Lasry-Lions approach, the problem typically reduces to the analysis of a forward-backward system of coupled nonlinear partial differential equations, called the {mean field game system}, composed of a Hamilton-Jacobi equation for the value function of the optimal control problem considered by a representative player and a Kolmogorov-Fokker-Planck equation describing the evolution of the population. We refer to \cite{CP2020} for an introduction to the theory.\\

The numerical analysis of the mean field game system was initiated by the works of Achdou and Capuzzo-Dolcetta \cite{AC2010} and Achdou, Camilli and Capuzzo-Dolcetta \cite{ACC2013}, where the authors considered a finite difference scheme. We refer to \cite{A2013,AL2020} for surveys on this topic. In addition, semi-Lagrangian schemes were proposed by \cite{CS2015,CS2018} and finite element approximations were recently considered by Osborne and Smears \cite{OS2024a,OS2025}. In these works, convergence results are qualitative in nature. On the side of quantitative convergence results, Bonnans, Liu and Pfeiffer \cite{BLP2023} obtained error estimates for a finite difference scheme, using a $\theta$-scheme for the time discretization and Osborne and Smears obtained error bounds for the finite element approximations of a stationary MFG system \cite{OS2025b,OS2025a}. A posteriori error estimates for finite element approximations were also studied by Osborne, Smears and Wells \cite{OSW2025} in the case of stationary MFG systems and by Smears and Wells \cite{SW2026} in the time-dependent case. These results relied on the so-called \emph{Lasry-Lions monotonicity condition}. In \cite{BLS2025,BLS2025a}, the authors proved error estimates on the finite element approximations of a stationary MFG system without the Lasry-Lions assumption by focusing on the class of \emph{stable solutions}. The notion of stable solutions was introduced by Briani and Cardaliaguet \cite{BC2018} and is characterized by the well-posedness of the linearized MFG system. \\

The main contribution of this paper is the derivation of a priori error estimates for semi-discrete in space $\mathbb{P}^1$ finite element approximations of stable solutions to the finite horizon MFG system with Dirichlet boundary conditions
\begin{equation}
\label{eq_mfg_parabolic_smooth}
 \begin{cases}
    - \partial_t u(t,x) - \Delta u(t,x) + H(x,Du(t,x)) = F[m](t,x) \quad & \textnormal{for } (t,x) \in (0,T) \times \Omega, \\
    \partial_t m(t,x) - \Delta m(t,x) - \diver \left ( m(t,x) H_p(x,Du(t,x)) \right )= 0 \quad & \textnormal{for } (t,x) \in (0,T) \times \Omega, \\
    u(t,x) = m(t,x) = 0 \quad & \textnormal{for }  (t,x) \in (0,T) \times \partial \Omega, \\
    u(T,\cdot) = u_T, \quad m(0,\cdot) = m_0 \quad & \textnormal{in } \Omega,
 \end{cases}
\end{equation}
where $\Omega \subset \RR^d$ is a bounded Lipschitz domain, $T> 0$ is a given constant, $H \colon \Omega \times \RR^d \to \RR$, $m_0, \, u_T \colon \Omega \to \RR$ and $F \colon L^2((0,T) \times \Omega) \to L^2((0,T) \times \Omega)$ are given functions and $H_p \colon \Omega \times \RR^d \to \RR^d$ denotes the gradient of $H$ with respect to the second variable.
The motivation for semi-discrete error estimates is that they serve as an intermediate step in the analysis of fully discrete schemes, isolating the spatial discretization error from the temporal one. We treat the spatial finite element error on its own, so that the remaining task for a fully discrete analysis reduces to the design and analysis of a suitable scheme for the forward-backward system of ODEs obtained after spatial discretization. In this way, our estimate provides a common basis from which various temporal discretization strategies may be analyzed.

We rely on the approach introduced in \cite{BLS2025} for the stationary MFG system, which uses the Brezzi-Rappaz-Raviart (BRR) approximation theorem \cite{BRR1980}. The main difference with \cite{BLS2025} lies in the fact that the parabolic nature of the systems makes the functional analysis required to apply the approach in \cite{BLS2025} much more involved. In particular, we had to make use of  the discrete $L^p$ maximal regularity estimates from \cite{L2019}. To the best of our knowledge, these are the first quantitative results about the convergence of finite element approximations for time-dependent MFG systems. Moreover, since our results concern stable solutions, we do not require the Lasry-Lions monotonicity conditions, allowing us to consider mean field games having multiple solutions. We emphasize that our results remain valid in the monotone setting, since the Lasry-Lions condition typically implies the stability of the unique solution. \\

Our first results deal with the natural semi-discrete approximation of the MFG system \eqref{eq_mfg_parabolic_smooth}. We prove in \cref{thm:error_estimate_parabolic} that for any stable solution $(u,m)$ to \eqref{eq_mfg_parabolic_smooth}, there exists a locally unique solution $(u_h,m_h)$ to the semi-discrete problem and we have the error estimate
\begin{multline*}
    \norm{u - u_h}{L^q} + \norm{Du - Du_h}{L^q} + \norm{m-m_h}{L^q} \\ \leq C \left( \norm{P_h m_0 - m_0^h}{L^q} + \norm{P_h u_T - u_T^h}{W^{1,q}} +   h^{2/q} \left( 1 + \module{\ln(h)}^{1-2/q} \right) \right).
   \end{multline*}
for $q> 2(d+2)$, where $P_h$ denotes the $L^2$ projection operator onto the finite element space. The logarithmic factor in the error estimates is a consequence of the use of $L^\infty$ estimates for the Ritz projection operator \cite{LV2016}. Moreover, if $(u,m)$ is sufficiently regular, we also prove the error bound
\begin{align*}
    & \norm{u - u_h}{L^q} + \norm{Du - Du_h}{L^q} + \norm{m - m_h}{L^q} \\
    & \quad \leq  C \left( \norm{P_h m_0 - m_0^h}{L^q} + h^{-1}\norm{P_h u_T - u_T^h}{L^q} + \inf_{v_h \in L^q(0,T;V_h)} \norm{u - v_h}{L^q} + \norm{Du - Dv_h}{L^q} \right) \\
   & \qquad + C h \left[ \inf_{\substack{\rho_h \in L^q(0,T;V_h)}} \norm{m - \rho_h}{L^q} + \norm{Dm - D\rho_h}{L^q} \right].
\end{align*}
Finally, under further regularity assumptions, we derive a error bound in quasi-optimal form in \cref{cor:quasi_optimal}. Since the natural semi-discrete scheme for \eqref{eq_mfg_parabolic_smooth} fails in general to preserve the nonnegativity of the solutions to the Fokker-Planck equation, we also prove analogous results for a semi-discrete scheme with mass lumping and stabilization by artificial viscosity on a strictly acute mesh in \cref{section:stabilized_scheme}. We prove that this modified scheme produces nonnegative solutions for the Fokker-Planck equation.\\

In practice, in order to deal with the nonlinear and coupled nature of the problem, the finite element discretization should be combined with an iterative method. In the context of this paper, the Newton method \cite{AP2012,CT2025,BLS2025,CZ2025}, the Fictitious Play algorithm \cite{CH2017,BC2018} and the Smoothed Policy Iteration method \cite{TS2024} are natural, since all these methods are known to converge locally to stable solutions. Since the convergence results are only local, it is also natural to incorporate a continuation method into the procedure.

The main restrictions in this paper are, first, strong regularity assumptions on the Hamiltonian (assumption \ref{h_parabolic_mfg}), similar to those in \cite{BLS2025}. This is required to ensure that the nonlinear mapping considered in the BRR approximation theorem is differentiable. Second, we assume the \emph{$L^p$ shift property} for $p$ large enough \ref{h_parabolic_regularity_q}. This assumption is satisfied when $\Omega$ has $C^{1,1}$ boundary, but imposes strong geometric restrictions when the domain is nonsmooth. In particular, for polygonal or polyhedral domains, it requires that interior angles are not too large. This assumption allows us to work with function spaces having sufficient integrability, which is used in the proof of \cref{prop_parabolic_mfg_diff}. We believe that investigating the possibility of relaxing these assumptions is an interesting question for future research.\\

The paper is structured as follows. In Section~\ref{section:notations}, we introduce notation and state the main assumptions on the Hamiltonian, coupling operator, and domain geometry. Section~\ref{section:mfg_parabolic} is devoted to general properties of the continuous mean field game system and the characterization of stable solutions in terms of an isomorphism property. In Section~\ref{section:parabolic_FEM}, we prove existence and a priori error estimates for semi-discrete finite element approximations of stable solutions to the MFG system. Finally, in  \cref{section:nemytskii} we recall some properties of Nemytskii operators which are used in the paper and \cref{section:proofs} contains some postponed proofs.

\section{Notations and assumptions}
\label{section:notations}
\paragraph{Notations.} We set $Q_T := (0,T) \times \Omega$. For $p \in [1,\infty]$ and a Banach space $X$, we denote by $L^p(0,T;X)$ the \emph{Bochner space} of strongly measurable mappings $u \colon (0,T) \to X$ such that
\[
 \norm{u}{L^p(X)} := \left( \int_0^T \norm{u(t)}{X}^p \, dt \right)^{1/p} < \infty
\]
if $p \in [1, \infty)$ and
\[
 \norm{u}{L^\infty(X)} := \inf \left \{ r \geq 0 :\, \module{\{t \in (0,T) : \norm{u(t)}{X} > r \}} = 0 \right \} < \infty
\]
otherwise. We refer to \cite[Chapter 1]{HNVW2016} for further details on this topic.

The \emph{parabolic Hölder seminorm} on $Q_T$ is defined by
\[
 [u]_{\alpha/2, \alpha} := \sup_{\substack{(t,x),(s,y) \in Q_T \\ (t,x) \neq (s,y)}} \frac{\module{u(t,x) - u(s,y)}}{\module{t - s}^{\alpha/2} + \module{x - y}^{\alpha}}.
\]
The \emph{parabolic Hölder norm} is then defined by
\[
 \norm{u}{\alpha/2, \alpha} := \module{u}_{\infty} + [u]_{\alpha/2, \alpha},
\]
for $\alpha \in (0,1]$, and the corresponding \emph{parabolic Hölder space}\index{parabolic Hölder space} by
\[
 C^{\alpha/2, \alpha}(Q_T) := \left \{ u \colon Q_T \to \RR : \, \norm{u}{\alpha/2, \alpha} < \infty \right \}.
\]

The \emph{parabolic Sobolev spaces} $\mathcal{H}^1_p(Q_T)$, $W^{0,1}_p(Q_T)$ and $W^{1,2}_p(Q_T)$, for $p \in [1,\infty]$ are defined as the functions $u \in L^p(0,T;W^{1,p}_0(\Omega))$ having finite $\norm{\cdot}{\mathcal{H}^1_p}$, $\norm{\cdot}{W^{0,1}_p}$ and $\norm{\cdot}{W^{1,2}_p}$ norms, respectively, where
\begin{align*}
 \norm{u}{\mathcal{H}^1_p} &:= \norm{u}{L^p(Q_T)} + \norm{Du}{L^p(Q_T)} + \norm{\partial_t u}{L^p(W^{-1,p})}, \\
 \norm{u}{W^{0,1}_p} & := \norm{u}{L^p(Q_T)} + \norm{Du}{L^p(Q_T)}, \\
 \norm{u}{W^{1,2}_p} & := \norm{u}{W^{0,1}_p} + \norm{D^2 u}{L^p(Q_T)} + \norm{\partial_t u}{L^p(Q_T)},
\end{align*}
and we recall that $W^{1,p}_0(\Omega)$ denotes the closure of $C^\infty_c(\Omega)$ with respect to the norm of the Sobolev space $W^{1,p}(\Omega)$.
\label{eq:parabolic_spaces}

\paragraph{Assumptions.}
Throughout this paper, we assume that the domain $\Omega$ satisfies both the \emph{uniform interior and exterior cone conditions}\index{cone condition}, \textit{i.e.}, there exist $r, \kappa > 0$ such that for every $x \in \partial \Omega$, there exist unit vectors $\xi_e = \xi_e(x)$ and $\xi_i = \xi_i(x)$ such that
\begin{equation*}
 \left\{ y \in B(x,r) :\, \xi_e \cdot \left(y - x \right) > 0 \textnormal{ and } \xi_e \cdot \left(y - x\right) < \kappa \module{y - x} \right \} \subset \Omega^\setcomp
\end{equation*}
and
\begin{equation*}
 \left\{ y \in B(x,r) :\, \xi_i \cdot \left(y - x \right) > 0 \textnormal{ and } \xi_i \cdot \left(y - x\right) < \kappa \module{y - x} \right \} \subset \Omega.
\end{equation*}
We now list the other assumptions that will be used below.
\begin{enumerate}[label={{\rm\bf(H\arabic*)}}]
 \item \label{h_parabolic_mfg} We assume the following.
 \begin{itemize}
  \item The Hamiltonian $H$ is of class $C^2$ with respect to the second variable and the functions $H$, $H_p$ and $H_{pp}$ are jointly continuous, where $H_p$ and $H_{pp}$ denote the gradient and the Hessian of $H$ with respect to the second variable, respectively. We also assume that there exists $C_H > 0$ such that
 \begin{gather}
  \module{H(x,p)} \leq C_H \left(1 + \module{p}^2 \right), \\
  \module{H_p(x,p)} \leq C_H \left(1 + \module{p} \right) , \label{eq_mfg_parabolic_smooth_growth} \\
  \module{H_{pp}(x,p)} \leq C_H,
 \end{gather}
 for every $(x,p) \in \Omega \times \RR^d$.

 \item For every $p \in [2, \infty)$, there exists $L_F > 0$ such that
    \begin{align*}
        \norm{F[m_1] - F[m_2]}{L^p(Q_T)} \leq L_F \norm{m_1 - m_2}{L^p(Q_T)} \quad \textnormal{for every } m_1, \, m_2 \in L^p(Q_T).
    \end{align*}

 \item The initial distribution $m_0$ and the terminal cost $u_T$ belong to $C^{\alpha}(\Omega)$, for some $\alpha \in (0,1)$, and $W^{2,\infty}(\Omega)$, respectively, and vanish on $\partial \Omega$.
 \end{itemize}

 \item \label{h_parabolic_regularity_q} The domain $\Omega$ is such that there exists $q > 2(d+2)$ such that the \emph{$L^p$-shift property} holds for every $1 < p \leq q/2$, \textit{i.e.}, there exists $C_p > 0$ such that
 \[
    \norm{v}{W^{2,p}} \leq C_p \left(\norm{f}{L^p(\Omega)} + \norm{v}{L^p(\Omega)} \right)
 \]
 whenever $v \in W^{1,p}_0(\Omega)$ is a weak solution to $\Delta v = f$ with homogeneous Dirichlet boundary conditions.

 \item \label{h_parabolic_mfg_diff} We assume that $F \colon L^p(Q_T) \to L^r(Q_T)$ is continuously differentiable for every $1 < r < p < \infty$.
\end{enumerate}

\begin{rem}\label{rem_mfg_parabolic_smooth_elliptic_reg}
Assumption \ref{h_parabolic_regularity_q} imposes strong restrictions on the geometry of the domain. It is known to hold when $\Omega$ has $C^{1,1}$ boundary \cite[Theorems  2.2.2.5 and 2.3.1.5]{G1985}. In the case of nonsmooth domains, the assumption does not hold in general if the domain is not convex. If $d = 2,\, 3$ and $\Omega$ has polygonal or polyhedral boundary, the assumption holds if interior angles are not too large \cite{D1988,D1992}. In particular, it is known that the assumption holds in box-shaped domains \cite{F1987}.
\end{rem}

\begin{rem}
 Assumption \ref{h_parabolic_mfg_diff} holds for instance if $F[m](t,x) = f(t,x,m(t,x))$ where $f \colon Q_T \times \RR \to \RR$ is a Carathéodory function which is $C^1$ with respect to the last variable and satisfies
 \begin{align*}
  \module{f(t,x,m)} & \leq C \left( 1 + \module{m} \right)\\
  \module{\partial_m f(t,x,m)} & \leq C
 \end{align*}
 for every $(t,x,m) \in Q_T \times \RR$ for some $C > 0$. Similarly, the assumption also holds if
 \[
  F[m](t,x) = f(t,x,k * m(t,x))
 \]
 where $k$ is a smooth convolution kernel and $f$ is as above.
\end{rem}

\section{General properties and stable solutions}
\label{section:mfg_parabolic}
This section is dedicated to the analysis of the MFG system \eqref{eq_mfg_parabolic_smooth}. We first study the well-posedness of the MFG system and we then turn to the characterization of its stable solutions in terms of an isomorphism property on the differential of some well-chosen mapping.
\subsection{Preliminary results on parabolic equations}
Let us start by recalling some results regarding the regularity of weak solutions to parabolic equations. The first one deals with the well-posedness of parabolic equations with first order terms in divergence form.
\begin{prop}[{\cite[Theorem 4.1 p.153]{LSU1968}}]\label{prop_parabolic_FP_wellposed}
 Let $b \in L^{d+2}(Q_T)$, $f,\, g \in L^2(Q_T)$, and $\rho_0 \in L^2(\Omega)$. Then there exists a unique weak solution $\rho \in \cH^1_2(Q_T)$ to
 \begin{equation}\label{eq_mfg_parabolic_smooth_FP_simple}
  \begin{cases}
   \partial_t \rho - \Delta \rho + \diver \left ( \rho b \right) = f + \diver(g) \quad & \textnormal{in } Q_T, \\
   \rho = 0 \quad & \textnormal{on } (0,T) \times \partial \Omega, \\
   \rho(0,\cdot) = \rho_0 \quad & \textnormal{in } \Omega,
  \end{cases}
 \end{equation}
 and there exists $C = C(\norm{b}{L^{d+2(Q_T)}}, T, d)$ such that
 \[
  \norm{\rho}{\cH^1_2} \leq C \left( \norm{g}{L^2(Q_T)} + \norm{f}{L^2(Q_T)} + \norm{\rho_0}{L^2(\Omega)} \right).
 \]
\end{prop}

We now recall the parabolic De Giorgi-Nash-Moser estimates.
\begin{prop}[De Giorgi-Nash-Moser, {\cite[Theorem 7.1 p.181]{LSU1968}}]\label{prop_parabolic_FP_DGNM}
 Let $b \in L^p(Q_T)$, $f \in L^{p/2}(Q_T)$, $g \in L^p(Q_T)$ for some $p > d+2$ and $\rho_0 \in L^\infty(\Omega)$. Then there exists $C_1 >0$, depending on $\norm{b}{L^p(Q_T)}$, $\Omega$, $T$ and $d$ such that the unique weak solution $\rho \in \cH^1_2(Q_T)$ to \eqref{eq_mfg_parabolic_smooth_FP_simple} satisfies
 \begin{gather*}
   \norm{\rho}{L^\infty(Q_T)} \leq C_1 \left ( \norm{\rho_0}{L^\infty(\Omega)} + \norm{f}{L^{p/2}(Q_T)} + \norm{g}{L^p(Q_T)} \right).
   \end{gather*}
  Moreover, if $\rho_0 \in C^{\alpha}(\Omega)$ for some $\alpha \in (0,1)$, then there exists $C_2>0$, depending on $\norm{\rho_0}{\alpha}$, $\alpha$, $C_1$, $\norm{f}{L^{p/2}(Q_T)}$ and $\norm{g}{L^p(Q_T)}$ such that
  \[
   \norm{\rho}{\beta} \leq C_2
  \]
  for some $\beta \in (0,\alpha]$.
\end{prop}

We recall the following fact about $L^p$ maximal regularity.
\begin{prop}[Maximal $L^p$ regularity, {\cite{HP1997}}]\label{prop_parabolic_W2p}
   Assume \ref{h_parabolic_regularity_q}  and let $2 \leq p \leq q/2$, where $q$ is fixed in \ref{h_parabolic_regularity_q}. Let also $f \in L^{p}(Q_T)$ and $v_0 \in W^{2,{p}}(\Omega) \cap W^{1,{p}}_0(\Omega)$. Then there exists a unique weak solution $v \in W^{1,2}_{p}(Q_T)$ to
   \begin{equation}
    \begin{cases}
     \partial_t v - \Delta v = f \quad & \textnormal{in } Q_T, \\
     v = 0 \quad & \textnormal{on } (0,T) \times \partial \Omega, \\
     v(0,\cdot) = v_0 \quad & \textnormal{in } \Omega,
    \end{cases}
   \end{equation}
  and there exists $C = C(p,T,d)>0$ such that
   \[
    \norm{v}{W^{1,2}_{p}} \leq C \left ( \norm{f}{L^{p}(Q_T)} + \norm{v_0}{W^{2,{p}}} \right).
   \]
\end{prop}

We conclude this section with a compact embedding lemma which will be useful below.
\begin{lem}\label{lem:parabolic_compact_embedding}
  Assume that $p > d+3$ and $p \geq 2d$. Then the embedding $W^{1,2}_{p/2}(Q_T) \hookrightarrow W^{0,1}_{p}(Q_T)$ is compact.
\end{lem}
\begin{proof}
 From the parabolic Sobolev inequality \cite[Theorem 6.11]{L1996}, we have the continuous embedding $W^{1,2}_{p/2}(Q_T) \hookrightarrow W^{0,1}_r(Q_T)$, with $r = \frac{p \left(d + p/2 \right)}{2d}$. Since $r \geq p$ if $p \geq 2d$, we obtain the continuous embedding $W^{1,2}_{p/2}(Q_T) \hookrightarrow W^{0,1}_{p}(Q_T)$. In order to prove the compactness of the embedding, we first notice that
 \[
  W^{1 + d/p,p/2}(\Omega) \hookrightarrow W^{1,p}(\Omega)
 \]
 according to \cite[Theorem 6.5]{DNPV2012}. Moreover, we have
 \[
  W^{1 + d/p,p/2}(\Omega) = \left(L^{p/2}(\Omega), W^{2,p/2}(\Omega) \right)_{\theta,p/2},
 \]
 where $\theta = \frac{p+d}{2p}$ and $\left( \cdot, \cdot \right)_{\theta,p/2}$ denotes the real interpolation functor \cite{L2018}. It follows that
 \[
  \left(L^{p/2}(\Omega), W^{2,p/2}(\Omega) \right)_{\theta,p/2} \hookrightarrow W^{1,p}(\Omega).
 \]
 Since $p > d+3$, we have $\frac{2}{p} - \left(1 - \theta \right) < \frac{1}{p}$ and it follows from the Aubin-Dubinskii lemma \cite{A2000} that the embedding $W^{1,2}_{p/2}(Q_T) \hookrightarrow W^{0,1}_p(Q_T)$ is compact.
\end{proof}

\subsection{Well-posedness}
This section is dedicated to the study of the well-posedness of the MFG system \eqref{eq_mfg_parabolic_smooth}. As a first step, we study the Hamilton-Jacobi equation independently of the MFG system. We then turn to the MFG system and conclude this section by considering the class of stable solutions to the system.

\subsubsection{Hamilton-Jacobi equations on nonsmooth domains}

In this section we study the well-posedness of Hamilton-Jacobi equation of the form
\begin{equation}\label{eq_mfg_parabolic_smooth_HJ}
\begin{cases}
 \partial_t u - \Delta u + H(x,Du) = f \quad & \textnormal{in $Q_T$}, \\
 u = 0 \quad & \textnormal{on } (0,T) \times \partial \Omega, \\
 u(0,\cdot) = u_0 \quad & \textnormal{on } \Omega,
\end{cases}
\end{equation}
where we assume that $f \in L^\infty(Q_T)$, $u_0 \in W^{2,\infty}(\Omega)$ and $H$ is a Carathéodory function satisfying
\begin{equation}\label{eq_mfg_parabolic_smooth_HJ_growth}
 \module{H(x,p)} \leq C_H\left( 1 + \module{p}^2 \right)
\end{equation}
for some $C_H > 0$. We also assume that $H$ is locally Lipschitz continuous with respect to its second variable and assume that
 \begin{equation}\label{eq_mfg_parabolic_smooth_HJ_clarke_growth}
  \xi \in \partial^C H(x,p) \Longrightarrow \module{\xi} \leq C_H \left( 1 + \module{p} \right),
 \end{equation}
 where $\partial^C H$ denotes Clarke's subdifferential with respect to the second variable\footnote{See \cref{section:nemytskii}.}.

\begin{defi}[Weak solution]
 We say that $u \in L^2(0,T;H^1(\Omega))$ with $\partial_t u \in L^2(0,T;H^{-1}(\Omega))$ is a \emph{weak subsolution} \index{weak solution} to \eqref{eq_mfg_parabolic_smooth_HJ} if
 \begin{multline}
  \int_0^t \langle \partial_t u(s), \phi(s) \rangle_{H^{-1},H^1} + \int_\Omega Du(s,x)\cdot D \phi(s,x) + H(x,Du(s,x)) \phi(s,x) \, dx ds \\ \leq \int_0^t \int_\Omega f(s,x) \phi(s,x) \, dx ds
 \end{multline}
 for all $t \in (0,T)$ and $\phi \in C^{\infty}_c([0,T] \times \Omega)$, with $\phi \geq 0$, $u(0,\cdot) \leq u_0$ almost everywhere and $u(t,\cdot) \leq 0$ on $\partial \Omega$ in the sense of traces for a.e. $t \in (0,T)$. Similarly, it is a \emph{weak supersolution} if
 \begin{multline}
  \int_0^t \langle \partial_t u(s), \phi(s) \rangle_{H^{-1},H^1} + \int_\Omega Du(s,x)\cdot D \phi(s,x) + H(x,Du(s,x)) \phi(s,x) \, dx ds \\ \geq \int_0^t \int_\Omega f(s,x) \phi(s,x) \, dx ds
 \end{multline}
 for all $t \in (0,T)$ and $\phi \in C^{\infty}_c([0,T] \times \Omega)$, $u(0,\cdot) \geq u_0$ almost everywhere and $u(t,\cdot) \geq 0$ on $\partial \Omega$ in the sense of traces for a.e. $t \in (0,T)$. Finally it is a \emph{weak solution} if it is both a weak sub- and supersolution to \eqref{eq_mfg_parabolic_smooth_HJ}.
\end{defi}

\begin{rem}
 Notice that if $u$ is a weak solution to \eqref{eq_mfg_parabolic_smooth_HJ}, then \cite[Theorem 2 p.273]{E2010} implies that $u(t,\cdot) \in H^1_0(\Omega)$ for a.e. $t \in (0,T)$, so that $u \in \cH^1_2(Q_T)$.
\end{rem}

The Hamilton-Jacobi equation \eqref{eq_mfg_parabolic_smooth_HJ} satisfies the following comparison principle, which implies uniqueness of solutions. Its proof is postponed to \cref{section:proof_comparison}.
\begin{prop}[Comparison principle]\label{prop_comparison_principle}
 Let $u$ and $v$ be weak sub- and supersolutions to \eqref{eq_mfg_parabolic_smooth_HJ}, respectively, and assume that $Du, \, Dv \in L^{d+2}(Q_T)$. Then $u \leq v$ almost everywhere. In particular there exists at most one weak solution to \eqref{eq_mfg_parabolic_smooth_HJ}  having $Du \in L^{d+2}(Q_T;\RR^d)$.
\end{prop}

Regarding existence of solutions to \eqref{eq_mfg_parabolic_smooth_HJ}, we have the following result. Its proof is based on the Leray-Schauder fixed point theorem and can be found in \cref{section:proof_HJ_wp}.
\begin{thm}\label{thm_mfg_parabolic_HJ_wp}
 In addition to \eqref{eq_mfg_parabolic_smooth_HJ_growth} and \eqref{eq_mfg_parabolic_smooth_HJ_clarke_growth}, assume that \ref{h_parabolic_regularity_q} holds. Then there exists a weak solution $u \in W^{1,2}_{q/2}(Q_T) \cap W^{0,1}_q(Q_T)$ to \eqref{eq_mfg_parabolic_smooth_HJ}, where $q$ is defined in \ref{h_parabolic_regularity_q}. Moreover, for any weak solution $u$ to \eqref{eq_mfg_parabolic_smooth_HJ} with $Du \in L^{d+2}(Q_T)$, we have
 \[
  \norm{u}{W^{1,2}_{q/2}} + \norm{u}{W^{0,1}_q} \leq C
 \]
 for some $C > 0$ depending on $d$, $T$, $C_H$, $\norm{f}{L^{q/2}(Q_T)}$ and $\norm{u_0}{W^{2,q/2}}$.
\end{thm}

\subsubsection{Well-posedness of the MFG system}

There exist many existence results for MFG systems similar to \eqref{eq_mfg_parabolic_smooth}. However, existing results seem to only consider periodic boundary conditions \cite{LL2007,CP2020}, assume the Lipschitz continuity of the Hamiltonian \cite{OS2025}, or smoothness of the boundary of the domain \cite{P2015,C2015}. The next result proves the existence of solutions to \eqref{eq_mfg_parabolic_smooth} in the case of possibly nonsmooth domains (satisfying \ref{h_parabolic_regularity_q}) and smooth Hamiltonians having up to quadratic growth.

\begin{defi}[Weak solutions]
 By a weak solution to \eqref{eq_mfg_parabolic_smooth}, we mean a pair $(u,m)$ with $u \in \cH^1_2(Q_T)$, $Du \in L^{r}(Q_T)$ with $r = \max \{4,d+2\}$ and $m \in \cH^1_2(Q_T) \cap L^{2(d+2)/d}(Q_T)$ such that

 \begin{multline}
  \int_0^t - \langle \partial_t u(s), \phi(s) \rangle_{H^{-1},H^1} + \int_\Omega Du(s,x) \cdot D \phi(s,x) + H(x,Du(s,x)) \phi(s,x)\, dx ds \\ = \int_0^t \int_{\Omega} F[m(s)](x) \phi(s,x)\, dx ds
 \end{multline}
for every $t\in (0,T)$ and $\phi \in W^{0,1}_2(Q_T)$,

\begin{multline}\label{eq:mfg_weak_sol_FP}
 \int_0^t \langle \partial_t m(s), \psi(s) \rangle_{H^{-1},H^1} + \int_{\Omega} Dm(s,x) \cdot D\psi(s,x) \, dx ds \\ + \int_0^t \int_\Omega m(s,x) H_p(x,Du(s,x)) \cdot D \psi(s,x)  \, dx ds = 0
\end{multline}
for every $t\in (0,T)$ and $\psi \in W^{0,1}_2(Q_T)$, $u(T,\cdot) = u_T$ and $m(0,\cdot) = m_0$.
\end{defi}

\begin{rem}
 Notice that the condition $u(T,\cdot) = u_T$ has a meaning for $u \in \cH^1_2(Q_T)$. Indeed, we recall that there exists a continuous embedding $\cH^1_2(Q_T) \hookrightarrow C([0,T],L^2(\Omega))$, see for instance \cite[Theorem 3.1 p. 19]{LM1972}. Moreover, since
 \[
    \frac{d}{2(d+2)} + \frac{1}{d+2} + \frac{1}{2} =1,
 \]
  the last integral in the left-hand side of \eqref{eq:mfg_weak_sol_FP} is finite.
\end{rem}

\begin{defi}[Lasry-Lions condition]
 We say that the \emph{Lasry-Lions condition}\index{Lasry-Lions condition} holds for the MFG system \eqref{eq_mfg_parabolic_smooth} if one of the following holds
 \begin{enumerate}[label={\rm{(\roman*)}}]
  \item the Hamiltonian $H$ is convex and $F \colon L^2(Q_T) \to L^2(Q_T)$ is strictly monotone;

  \item the Hamiltonian $H$ is strictly convex and $F\colon L^2(Q_T) \to L^2(Q_T)$ is monotone.
 \end{enumerate}
 Moreover, we say that the \emph{strong Lasry-Lions condition}\index{Lasry-Lions condition!strong} holds if one of the following holds
 \begin{enumerate}[label={\rm{(\roman*)}}]
  \item the Hamiltonian $H$ is convex and $F \colon L^2(Q_T) \to L^2(Q_T)$ is strongly monotone;

  \item the Hamiltonian satisfies $H_{pp}(x,p) \geq C_H^{-1} I$ and $F \colon L^2(Q_T) \to L^2(Q_T)$ is monotone.
\end{enumerate}
\end{defi}

We are going to express solutions to \eqref{eq_mfg_parabolic_smooth} as zeros of some abstract mapping $\Upsilon$. Keeping things formal for now, we define the linear operators $S_I$, $S_T$, $S_{IS}$ and $S_{TS}$ by
 \begin{equation}\label{eq_mfg_parabolic_smooth_SI}
  S_I(\rho_0) = \rho, \quad \textnormal{where} \quad \begin{cases}
                                                \partial_t \rho - \Delta \rho = 0 \quad & \textnormal{in } Q_T, \\
                                                \rho = 0 \quad & \textnormal{on } (0,T) \times \partial \Omega, \\
                                                \rho(0,\cdot) = \rho_0 \quad & \textnormal{in } \Omega,
                                               \end{cases}
 \end{equation}
 \begin{equation}\label{eq_mfg_parabolic_smooth_ST}
  S_T(v_T) = v, \quad \textnormal{where} \quad \begin{cases}
                                                - \partial_t v - \Delta v = 0 \quad & \textnormal{in } Q_T, \\
                                                v = 0 \quad & \textnormal{on } (0,T) \times \partial \Omega, \\
                                                v(T,\cdot) = v_T \quad & \textnormal{in } \Omega,
                                               \end{cases}
 \end{equation}
  \begin{equation}\label{eq_mfg_parabolic_smooth_SIS}
  S_{IS}(g) = \rho, \quad \textnormal{where} \quad \begin{cases}
                                                \partial_t \rho - \Delta \rho = g \quad & \textnormal{in } Q_T, \\
                                                \rho = 0 \quad & \textnormal{on } (0,T) \times \partial \Omega, \\
                                                \rho(0,\cdot) = 0\quad & \textnormal{in } \Omega,
                                               \end{cases}
 \end{equation}
 and
\begin{equation}\label{eq_mfg_parabolic_smooth_STS}
  S_{TS}(f) = v, \quad \textnormal{where} \quad \begin{cases}
                                                - \partial_t v - \Delta v = f \quad & \textnormal{in } Q_T, \\
                                                v = 0 \quad & \textnormal{on } (0,T) \times \partial \Omega, \\
                                                v(T,\cdot) = 0 \quad & \textnormal{in } \Omega.
                                               \end{cases}
 \end{equation}
 Notice that, at least formally, $(u,m)$ is a solution to \eqref{eq_mfg_parabolic_smooth} if and only if
 \[
  \begin{pmatrix} u \\ m \end{pmatrix} = \begin{pmatrix} S_T u_T + S_{TS} \left (F[m] - H(\cdot, Du)\right) \\  S_I m_0 + S_{IS} \left( \diver \left (m H_p(\cdot,Du) \right) \right) \end{pmatrix}.
 \]
 Therefore, setting
 \begin{equation}\label{eq_mfg_parabolic_smooth_def_S}
  \cS_1 := \begin{pmatrix} S_T & 0 \\ 0 & S_I \end{pmatrix}, \quad \cS_2 := \begin{pmatrix} S_{TS} & 0 \\ 0 & S_{IS} \end{pmatrix}
 \end{equation}
 and
 \begin{equation}\label{eq_mfg_parabolic_smooth_def_R}
  \cR_1(v,\rho) := \begin{pmatrix} u_T \\ m_0 \end{pmatrix}, \quad \cR_2(v,\rho) := \begin{pmatrix} F[\rho] - H(\cdot, Dv) \\ \diver \left (\rho H_p(\cdot,Dv) \right) \end{pmatrix},
 \end{equation}
 we may define
\begin{equation}\label{eq_mfg_parabolic_smooth_def_Upsilon}
  \Upsilon(v,\rho) := \left( I - \cS_1 \cR_1 - \cS_2 \cR_2 \right)(v,\rho),
 \end{equation}
 so that $(u,m)$ is a solution to \eqref{eq_mfg_parabolic_smooth} if and only if
 \[
  \Upsilon(u,m) = 0.
 \]

We turn to the question of existence and uniqueness of solutions to \eqref{eq_mfg_parabolic_smooth}. Since the proof is very similar to the standard one, we postpone it to \cref{section:proof_mfg_wp}.
\begin{thm}\label{thm_mfg_wp}
 Assume \ref{h_parabolic_mfg}, \ref{h_parabolic_regularity_q} and that $F$ has bounded image in $L^\infty(Q_T)$. Then there exists a weak solution to \eqref{eq_mfg_parabolic_smooth}. Moreover, if the Lasry-Lions condition holds, then this solution is unique.
\end{thm}

\subsection{Stable solutions}

The following definition was first introduced in \cite{BC2018}. Throughout this section, we assume that \ref{h_parabolic_mfg_diff} holds.

\begin{defi}[Stable solutions]
 We say that a weak solution $(u,m)$ to \eqref{eq_mfg_parabolic_smooth} is \emph{stable} if $(v,\rho) = (0,0)$ is the unique weak solution in $\left(\cH^1_2(Q_T) \cap W^{0,1}_{d+2}(Q_T) \right) \times \left(\cH^1_2(Q_T) \cap L^{2(d+2)/d}(Q_T) \right)$ to the linearized system
 \begin{equation}\label{eq_mfg_parabolic_smooth_linearized}
  \begin{cases}
   - \partial_t v - \Delta v + H_p(x,Du) \cdot Dv = dF[m](\rho) \quad & \textnormal{in } Q_T, \\
   \partial_t \rho - \Delta \rho - \diver \left (\rho H_p(x,Du) \right ) = \diver \left (m H_{pp}(x,Du)Dv \right) \quad & \textnormal{in } Q_T, \\
   v = \rho = 0 \quad & \textnormal{on } (0,T) \times \partial \Omega,\\
   v(T,\cdot) = 0, \quad \rho(0,\cdot) = 0 \quad & \textnormal{in } \Omega,
  \end{cases}
 \end{equation}
 in the sense that
 \begin{multline}
  \int_0^t - \langle \partial_t v(s), \phi(s) \rangle_{H^{-1},H^1_0} + \int_\Omega Dv(s,x) \cdot D \phi(s,x) + H_p(x,Du(s,x)) \cdot Dv(s,x) \phi(s,x)\, dx ds \\ = \int_0^t \int_{\Omega} dF[m(s)](\rho(s))(x) \phi(s,x)\, dx ds
 \end{multline}
for every $t\in (0,T)$ and $\phi \in W^{0,1}_2(Q_T)$,
\begin{multline}
 \int_0^t  \langle \partial_t \rho(s), \psi(s) \rangle_{H^{-1},H^1_0} + \int_\Omega D\rho(s,x) \cdot D \psi(s,x) + \rho(s,x) H_p(x,Du(s,x)) \cdot D\psi(s,x)\, dx ds \\ = - \int_0^t \int_\Omega \left(m(s,x) H_{pp}(x,Du(s,x)) Dv(s,x) \right) \cdot D\psi(s,x) \ dx ds
\end{multline}
for every $t\in (0,T)$ and $\psi \in W^{0,1}_2(Q_T)$, $v(T,\cdot) = 0$ and $\rho(0,\cdot) = 0$.
\end{defi}

Consider the following spaces
\begin{equation}\label{eq_mfg_parabolic_smooth_spaces}
 \begin{gathered}
 X := W^{0,1}_{q}(Q_T) \times L^q(Q_T) \quad Y_1 := W^{2,\infty}(\Omega) \times L^\infty(\Omega) \\  Y_2 := L^{q/2}(Q_T) \times L^{(q-\eta)/2}(0,T;W^{-1,{(q-\eta)/2}}(\Omega)),
\end{gathered}
\end{equation}
where $q > 2(d+2)$ and $ 0 <\eta < q - 2(d+2)$ is such that $\frac{(q-\eta)(d + (q-\eta))}{d} > q$.
One may verify that with this choice of spaces, the mappings $\cR_1 \colon X \to Y_1$ and $\cR_2 \colon X \to Y_2$ defined in \eqref{eq_mfg_parabolic_smooth_def_R} are well-defined. Moreover, from Propositions \ref{prop_parabolic_FP_wellposed}, \ref{prop_parabolic_FP_DGNM} and \ref{prop_parabolic_W2p}, and using the fact that
\[
 L^{(q-\eta)/2}(0,T;W^{-1,{(q-\eta)/2}}(\Omega)) \hookrightarrow L^2 \left(0,T; H^{-1}(\Omega) \right),
\]
we also have
\begin{gather*}
   S_T \in \cL(W^{2,\infty}(\Omega), W^{1,2}_{q/2}(Q_T)),\, S_I \in \cL(L^\infty(\Omega), \cH^1_2(Q_T) \cap L^\infty(Q_T)), \\
   S_{TS} \in \cL(L^{q/2}(Q_T), W^{1,2}_{q/2}), \, S_{IS} \in \cL(L^{(q-\eta)/2}(0,T;W^{-1,{(q-\eta)/2}}(\Omega)), \cH^1_2(Q_T) \cap L^\infty(Q_T)).
\end{gather*}
It follows that the mapping $\Upsilon \colon X \to X$ defined in \eqref{eq_mfg_parabolic_smooth_def_Upsilon} is well-defined.
Let us recall the following definition.
\begin{defi}[Strict differentiability]
 Let $X$ and $Y$ be Banach spaces and let $F \colon X \to Y$. We say that $F$ is \emph{strictly differentiable}\index{strict differentiability} at $\bar x \in X$ if $F$ is Fréchet differentiable at $\bar x$ and if, for every $\epsilon > 0$, there exists $\delta > 0$ such that
 \begin{equation}\label{eq:strict_differentiability}
  \norm{F(x) - F(y) - dF[\bar x](x - y)}{Y} \leq \epsilon \norm{x - y}{X} \quad \textnormal{for all } x,y \in  B_X(\bar x, \delta).
 \end{equation}
\end{defi}
It is well-known that if $F \colon X \to Y$ is $C^1$ on a neighborhood of $\bar x \in X$, then it is strictly differentiable at $\bar x$. Notice that if $F$ is strictly differentiable at $\bar x$, then we can define the function $c \colon \RR_+^* \to \RR_+$ by
\[
 c(\delta):= \inf \left \{ \epsilon > 0 \textnormal{ such that \eqref{eq:strict_differentiability} holds} \right\}.
\]
Clearly $c$ is nondecreasing, satisfies $c(0^+) = 0$ and
\[
 \norm{F(x) - F(y) - dF[\bar x](x - y)}{Y} \leq c(\delta) \norm{x - y}{X} \quad \textnormal{for all } x,y \in  B_X(\bar x, \delta).
\]

We now prove the strict differentiability of $\Upsilon$.
\begin{prop}\label{prop_parabolic_mfg_diff}
 Let $X$ and $Y_2$ be as in \eqref{eq_mfg_parabolic_smooth_spaces}. Assume \ref{h_parabolic_mfg}, \ref{h_parabolic_regularity_q} and \ref{h_parabolic_mfg_diff} and let $(u,m) \in X$. Then the mapping $\cR_2 \colon X \to Y_2$ is strictly differentiable at $(u,m)$ with
 \begin{equation}\label{eq:diff_R_2}
  d\cR_2[u,m](v,\rho) = \begin{pmatrix}
                         dF[m](\rho) - H_p(\cdot,Du) \cdot Dv \\
                         \diver \left (\rho H_p(\cdot,Du) + m H_{pp}(\cdot,Du)Dv \right )
                        \end{pmatrix}.
 \end{equation}
\end{prop}

\begin{proof}
 Let $\epsilon > 0$. From \ref{h_parabolic_mfg_diff} there exists $\delta_1 > 0$ such that
 \[
  \norm{F[m_1] - F[m_2] - dF[m](m_1 - m_2)}{L^{q/2}(Q_T)} \leq \frac{\epsilon}{5} \norm{m_1 - m_2}{L^q(Q_T)}
 \]
 for every $m_1,\, m_2 \in B_{L^q(Q_T)}(m, \delta_1)$. From \ref{h_parabolic_mfg} and \cref{thm:nemytskii_lebesgue}, we know that the Nemytskii operator
 \begin{equation}\label{eq_mfg_parabolic_smooth_nemytskii_H}
  \mathfrak{H} \colon \left \{ \begin{array}{l}
                                 L^q(Q_T;\RR^d) \to L^{q/2}(Q_T) \\
                                 w \mapsto H(x,w(t,x))
                               \end{array} \right.
 \end{equation}
 is continuously differentiable. In particular there exists $\delta_2 > 0$ such that
 \[
  \norm{\mathfrak{H}[Du_1] - \mathfrak{H}[Du_2] - \mathfrak{H_p}[Du]\cdot(Du_1 - Du_2)}{L^{q/2}(Q_T)} \leq \frac{\epsilon}{5} \norm{Du_1 - Du_2}{L^q(Q_T)}
 \]
 for every $u_1, u_2 \in B_{W^{0,1}_q(Q_T)}(u,\delta_2)$, where
\begin{equation}\label{eq_mfg_parabolic_smooth_nemytskii_Hp}
  \mathfrak{H_p} \colon \left \{ \begin{array}{l}
                                 L^q(Q_T;\RR^d) \to L^q(Q_T;\RR^d) \\
                                 w \mapsto H_p(x,w(t,x))
                               \end{array} \right. .
 \end{equation}
 Define the Nemytskii operator
 \begin{equation}\label{eq_mfg_parabolic_smooth_nemytskii_Hpp}
  \mathfrak{H_{pp}} \colon \left \{ \begin{array}{l}
                                 L^q(Q_T;\RR^d) \to L^\infty(Q_T;\RR^{d \times d}) \\
                                 w \mapsto H_{pp}(x,w(t,x))
                               \end{array} \right. .
 \end{equation}
 Then,
 \begin{align*}
   \|m_1 \mathfrak{H_p}[Du_1] & - m_2 \mathfrak{H_p}[Du_2] - (m_1 - m_2) \mathfrak{H_p}[Du] - m \mathfrak{H_{pp}}[Du](Du_1 - Du_2)\|_{L^{(q-\eta)/2}(Q_T)} \\
  & \quad \leq \norm{(m_1 - m_2) \left (\mathfrak{H_p}[Du_1] - \mathfrak{H_p}[Du] \right) }{L^{(q-\eta)/2}(Q_T)} \\ & \qquad + \norm{(m_2 - m) \left ( \mathfrak{H_p}[Du_1] - \mathfrak{H_p}[Du_2] \right)}{L^{(q - \eta)/2}(Q_T)} \\
  & \qquad + \norm{m \left( \mathfrak{H_p}[Du_1] - \mathfrak{H_p}[Du_2] - \mathfrak{H_{pp}}[Du](Du_1 - Du_2)\right)}{L^{(q-\eta)/2}(Q_T)}.
 \end{align*}
Moreover, there exists $\delta_3,\, \delta_4 > 0$ such that
\begin{align*}
 \|(m_1 - m_2) (\mathfrak{H_p}[Du_1] - & \mathfrak{H_p}[Du]) \|_{L^{(q-\eta)/2}(Q_T)} \\
 & \leq C_\eta \norm{m_1 - m_2}{L^q(Q_T)} \norm{\mathfrak{H_p}[Du_1] - \mathfrak{H_p}[Du]}{L^q(Q_T)} \\
 & \leq C_\eta C_H \norm{m_1 - m_2}{L^q(Q_T)} \norm{Du_1 - Du}{L^q(Q_T)} \\
 & \leq \frac{\epsilon}{5} \norm{m_1 - m_2}{L^q(Q_T)},
\end{align*}
\begin{align*}
\norm{(m_2 - m) \left ( \mathfrak{H_p}[Du_1] - \mathfrak{H_p}[Du_2] \right)}{L^{(q - \eta)/2}(Q_T)} & \leq C_\eta C_H \norm{m_2 - m}{L^q(Q_T)} \norm{Du_1 - Du_2}{L^q(Q_T)} \\
& \leq \frac{\epsilon}{5} \norm{Du_1 - Du_2}{L^q(Q_T)} ,
\end{align*}
and
\begin{multline*}
 \norm{m \left( \mathfrak{H_p}[Du_1] - \mathfrak{H_p}[Du_2] - \mathfrak{H_{pp}}[Du](Du_1 - Du_2)\right)}{L^{(q-\eta)/2}(Q_T)} \\
  \leq C_\eta \norm{m}{L^q(Q_T)} \norm{\mathfrak{H_p}[Du_1] - \mathfrak{H_p}[Du_2] - \mathfrak{H_{pp}}[Du](Du_1 - Du_2)}{L^{q - \eta}(Q_T)}
\end{multline*}
for $u_1,\, u_2 \in B_{W^{0,1}_{q}}(u,\delta_3)$ and $m_1,\, m_2 \in B_{L^q(Q_T)}(m,\delta_4)$. Without loss of generality, we may assume that $m \neq 0$. It then also follows from \cref{thm:nemytskii_lebesgue} that $\mathfrak{H_p} \colon L^q(Q_T;\RR^d) \to L^{q - \eta}(Q_T;\RR^d)$ is continuously differentiable and there exists $\delta_5 > 0$ such that
 \begin{multline*}
  \norm{\mathfrak{H_p}[Du_1] - \mathfrak{H_p}[Du_2] - \mathfrak{H_{pp}}[Du](Du_1 - Du_2)}{L^{q - \eta}(Q_T)} \\ \leq \frac{\epsilon}{5 C_\eta \norm{m}{L^q(Q_T)}}\norm{Du_1 - Du_2}{L^q(Q_T)}
 \end{multline*}
 for every $u_1, u_2 \in B_{W^{0,1}_q(Q_T)}(u,\delta_5)$. In conclusion, we have
 \begin{align*}
  \norm{\cR_2(u_1,m_1) - \cR_2(u_2,m_2) - d\cR_2[u,m](u_1 - u_2, m_1 - m_2)}{Y_2} \leq \epsilon \norm{(u_1 - u_2, m_1 - m_2)}{X}
 \end{align*}
 for every $(u_i,m_i) \in B_{X}((u,m),\delta)$, where $\delta = \min \left \{ \delta_i :\, 1 \leq i \leq 5 \right \}$ and
 \begin{equation}\label{eq:diff_R}
  d\cR_2[u,m](v,\rho) = \begin{pmatrix}
                         dF[m](\rho) - \mathfrak{H_p}[Du] \cdot Dv \\
                         \diver \left (\rho \mathfrak{H_p}[Du] + m \mathfrak{H_{pp}}[Du]Dv \right )
                        \end{pmatrix}.
 \end{equation}
 This concludes the proof.
 \end{proof}

 We now prove that the stability of a solution to the MFG system is related to the invertibility of the differential of $\Upsilon$.

 \begin{thm}[Isomorphism property]\label{thm_mfg_parabolic_isom}
   Assume \ref{h_parabolic_mfg}, \ref{h_parabolic_regularity_q} and \ref{h_parabolic_mfg_diff}, let $X,\, Y_1$ and $Y_2$ be as in \eqref{eq_mfg_parabolic_smooth_spaces}, with $q$ fixed in \ref{h_parabolic_regularity_q}, and let $(u,m) \in X$ be a stable solution to \eqref{eq_mfg_parabolic_smooth}. Then the mapping $\Upsilon \colon X \to X$ defined in \eqref{eq_mfg_parabolic_smooth_def_Upsilon}, with $\cR_1 \colon X \to Y_1$ and $\cR_2 \colon X \to Y_2$ as in \eqref{eq_mfg_parabolic_smooth_def_R}, is strictly differentiable at $(u,m)$ and $d\Upsilon[u,m]$ is an isomorphism on $X$, with
  \[
    d\Upsilon[u,m](v,\rho) = \left (I - \cS_2 \circ d\cR_2[u,m] \right)(v,\rho).
  \]
 \end{thm}

 \begin{proof}
   Recall that $d\cR_2[u,m] \in \cL(X,Y_2)$. Moreover, we have
   \[
    \cS_2 \in \cL \left (Y_2, W^{1,2}_{q/2}(Q_T) \times \left(L^\infty(Q_T) \cap \cH^1_2(Q_T) \right) \right).
   \]
   The strict differentiability of $\Upsilon$ then directly follows from \cref{prop_parabolic_mfg_diff}.

   We now prove the invertibility of $d\Upsilon[u,m]$ when $(u,m)$ is a stable solution.
   Assume that $(v,\rho) \in \ker_{X} \left( d \Upsilon[u,m] \right)$. In particular
   \[
    (v,\rho) =  \cS_2 \circ d\cR_2[u,m](v,\rho),
   \]
   and therefore
   \[
    (v,\rho) \in \left(\cH^1_2(Q_T) \cap W^{0,1}_{d+2}(Q_T) \right) \times  \left(\cH^1_2(Q_T) \cap L^{2(d+2)/d}(Q_T) \right).
   \]
    Moreover $(v,\rho)$ is a weak solution to \eqref{eq_mfg_parabolic_smooth_linearized}. Since we assume that $(u,m)$ is a stable solution to \eqref{eq_mfg_parabolic_smooth}, we conclude that $(v,\rho) = (0,0)$. This proves the injectivity of $d\Upsilon[u,m]$ in $X$. We now prove its surjectivity. From \cref{lem:parabolic_compact_embedding} we know that the embedding
   \[
     W^{1,2}_{q/2}(Q_T)  \hookrightarrow W^{0,1}_q(Q_T)
   \]
   is compact. Moreover, from \cref{prop_parabolic_FP_DGNM} and the Arzela-Ascoli theorem, we deduce that $S_{IS}$ maps bounded sets into relatively compact subsets of $L^{q}(Q_T)$. It follows that the linear operator $\cS_2 \circ d\cR_2[u,m]$ is compact as an element of $\cL(X)$. The surjectivity of $d\Upsilon[u,m]$ then follows from Fredholm's alternative \cite[Theorem 6.6]{B2011}.
 \end{proof}

\section{Semi-discrete error estimates for finite element approximations of stable solutions}
\label{section:parabolic_FEM}
We fix $d \leq 3$, we assume the boundary of $\Omega$ is polygonal if $d=2$ and polyhedral if $d=3$, with $\Omega$ convex. For $h >0$, let $\mathcal{T}_h$ be a quasi-uniform triangulation of $\Omega$ (see \cite[Definition 4.4.13]{BS2008}), where $h = \max_{K \in \cT_h} \diam(K)$. Let also $V_h \subset W^{1,\infty}_0(\Omega)$ be the associated finite element space induced by $\PP^1$-Lagrange finite elements.

We look for a semi-discrete approximation of a weak solution $(u,m)$ to \eqref{eq_mfg_parabolic_smooth}. An element $(u_h,m_h) \in H^1(0,T;V_h) \times H^1(0,T;V_h)$ is a solution to the semi-discrete MFG system if
\begin{multline}\label{eq:mfg_semi-discrete_HJB}
 \int_0^t \int_\Omega - \partial_t u_h(s,x) \phi_h(s,x) + Du_h(s,x) \cdot D\phi_h(s,x) + H(x,Du_h(s,x)) \phi_h(s,x) \, dx ds \\ = \int_0^t \int_\Omega F[m_h(s)](x) \phi_h(s,x) \, dx ds
\end{multline}
for all $t \in (0,T)$ and $\phi_h \in L^2(0,T;V_h)$,
\begin{multline}\label{eq:mfg_semi-discrete_FP}
 \int_0^t \int_\Omega \partial_t m_h(s,x) \psi_h(s,x) + Dm_h(s,x) \cdot D\psi_h(s,x)\, dx ds \\ + \int_0^t \int_\Omega m_h(s,x) H_p(x,Du_h(s,x)) \cdot D\psi_h(s,x) \, dx ds = 0
\end{multline}
for all $t \in (0,T)$ and $\psi_h \in L^2(0,T;V_h)$, and if $u_h(T,\cdot) = u_T^h \in V_h$ and $m_h(0,\cdot) = m_0^h \in V_h$.

\begin{rem}
  In general the approximate density $m_h$ may fail to remain nonnegative. This issue will be addressed in \cref{section:stabilized_scheme} below.
\end{rem}

We start this section by recalling some facts about finite element approximations. We then prove some estimates on the linear semi-discrete problems and conclude with our result on semi-discrete approximations of stable solutions to \eqref{eq_mfg_parabolic_smooth}.

\subsection{Some results on finite element approximations}
In what folows, we denote by $P_h \in \cL(L^2(\Omega))$ the \emph{$L^2$ projection} onto $V_h$, i.e.
\[
 P_h u := \argmin_{u_h \in V_h} \norm{u - u_h}{L^2(\Omega)}.
\]We recall the following fact.
\begin{prop}\label{prop_parabolic_mfg_proj}
   For every $1 < p < \infty$, there exists $C_p > 0$ such that
   \begin{gather*}
    \norm{P_h u}{L^p(\Omega)} \leq C_p \norm{u}{L^p(\Omega)} \quad \textnormal{for all } u \in L^p(\Omega), \\
    \norm{P_h u}{W^{1,p}(\Omega)} \leq C_p \norm{u}{W^{1,p}(\Omega)} \quad \textnormal{for all } u \in W_0^{1,p}(\Omega).
   \end{gather*}
   Moreover,
   \[
    \norm{u - P_h u}{L^2(\Omega)} \leq C h \norm{u}{H^1} \quad \textnormal{for all } u \in H_0^1(\Omega).
   \]
\end{prop}
\begin{proof}
 The first two inequalities are proved in \cite{CT1987}. The last inequality can be found in \cite[Proposition 1.134]{EG2004}.
\end{proof}

The \emph{Ritz projection}\index{Ritz projection} on $V_h$ is defined by $R_h u = u_h$, where $u_h$ is the unique solution in $V_h$ to
\[
 \int_{\Omega} Du_h \cdot Dv_h \, dx = \int_{\Omega} Du \cdot D v_h \quad \textnormal{for all } v_h \in V_h.
\]
Note that $R_h$ is well defined according to the Lax-Milgram theorem. We also recall the following result.
\begin{prop}(Stability of the Ritz projection)\label{prop_parabolic_mfg_ritz}
 Assume \ref{h_parabolic_regularity_q}. Then, for every $1 < p \leq \infty$, we have
 \[
  \norm{R_h u}{W^{1,p}} \leq C \norm{u}{W^{1,p}} \quad \textnormal{for all } u \in W^{1,p}_0(\Omega).
 \]
 In addition, we also have
 \begin{equation}\label{eq:ritz_max_norm}
  \norm{R_h u}{L^\infty(\Omega)} \leq C \module{\ln \left( h \right)} \norm{u}{L^\infty(\Omega)} \quad \textnormal{for all } u \in H^2(\Omega) \cap H^1_0(\Omega).
 \end{equation}
\end{prop}
\begin{proof}
 The first inequality is a consequence of \cite[Theorem 8.5.3]{BS2008}. The second inequality is proved in \cite[Theorem 12]{LV2016}.
\end{proof}

Let $V$ and $W$ be Banach spaces and $A \in \cL(V,W)$. We may then consider the extension $A \in \cL \left( L^p(0,T;V), L^p(0,T;W) \right)$ by setting $\left(Av \right)(t) = A\left(v(t) \right)$ for $t\in (0,T)$. This fact will be used repeatedly below.

\subsection{Estimates on linear semi-discrete problems}

For $f \in L^2(0,T;V_h')$ we consider the problem: find $v_h \in H^1(0,T;V_h)$ such that $v_h(0) = v_0^h \in V_h$ and
\begin{equation}\label{eq_mfg_parabolic_smooth_semi-discrete}
 \int_0^t \int_{\Omega} \partial_t v_h(s,x) \phi_h(s,x) +  D v_h(s,x) \cdot D \phi_h(s,x) \, dx ds = \int_0^t \langle f(s), \phi_h \rangle_{V_h',V_h} \, ds
\end{equation}
for all $t \in (0,T)$ and $\phi_h \in L^2(0,T;V_h)$.
Existence and uniqueness of a solution to problem \eqref{eq_mfg_parabolic_smooth_semi-discrete} may be proved by using the Cauchy-Lipschitz theory, see \cite[Proposition 66.2]{EG2021} for instance.
We have the following $L^2$ error estimates for solutions to \eqref{eq_mfg_parabolic_smooth_semi-discrete}.
\begin{thm}[{\cite[Theorems 3.2, 3.3 and 3.5]{CH2002}}]\label{thm_mfg_parabolic_semi-discrete_heat_error_2}
 Let $f \in L^2(0,T;H^{-1}(\Omega))$, $v_0 \in L^2(\Omega)$ and $v_0^h = P_h v_0$. Let $v \in \cH^1_2(Q_T) $ be the unique weak solution to
 \begin{equation}\label{eq_mfg_parabolic_smooth_heat2}
  \begin{cases}
   \partial_t v - \Delta v = f \quad & \textnormal{in } Q_T, \\
   v = 0 \quad & \textnormal{on } (0,T) \times \partial \Omega, \\
   v(0, \cdot) = v_0 \quad & \textnormal{on } \Omega
  \end{cases}
 \end{equation}
  and $v_h$ be the unique solution to \eqref{eq_mfg_parabolic_smooth_semi-discrete}. Then
  \begin{align*}
    \norm{v - v_h}{L^2(Q_T)} & \leq Ch \left( \norm{f}{L^2(H^{-1})} + \norm{v_0}{L^2(\Omega)} \right), \\
    \norm{v - v_h}{\cH^1_2} & \leq C h \left( \norm{f}{L^2(Q_T)} + \norm{v_0}{H^1} \right).
  \end{align*}
\end{thm}

The following is a consequence of discrete $L^p$ maximal regularity estimates.
\begin{thm}[{\cite[Theorem 2.1 and Corollary 2.2]{L2019}}]\label{thm_mfg_parabolic_semi-discrete_heat_error_p}
 Let $f \in L^2(0,T;H^{-1}(\Omega))$, $v_0 \in L^2(\Omega)$ and $v_0^h \in V_h$. Let $v \in \cH^1_2(Q_T)$ be the unique weak solution to
 \[
  \begin{cases}
   \partial_t v - \Delta v = f \quad & \textnormal{in } Q_T, \\
   v = 0 \quad & \textnormal{on } (0,T) \times \partial \Omega, \\
   v(0, \cdot) = v_0 \quad & \textnormal{on } \Omega
  \end{cases}
 \]
  and $v_h$ be the unique solution to \eqref{eq_mfg_parabolic_smooth_semi-discrete}.
  If $v \in L^p(Q_T)$ for some $1 < p < \infty$, then
 \begin{align*}
  \norm{v_h - P_h v}{L^p(Q_T)} & \leq C \left( \norm{v - R_h v}{L^p(Q_T)} + \norm{P_h v_0 - v_0^h}{L^p(\Omega)} \right),
 \end{align*}
 where $C >0$ is independent of $v$, $v_h$, $f$, $v_0$ and $v_0^h$. Moreover, if $f = 0$, we also have
 \begin{align*}
  \norm{v_h}{L^p(Q_T)} & \leq C \norm{v_0^h}{L^p(\Omega)}.
 \end{align*}
\end{thm}

In particular, we have the following $W^{0,1}_p$-stability estimate.
\begin{cor}\label{cor:W_0_1_p_stability}
 Assume that $\Omega$ is convex and let $1 < p < \infty$. With the notations of \cref{thm_mfg_parabolic_semi-discrete_heat_error_p}, if $f \in L^p(0,T;W^{-1,p})$ and $v_0 \in W_0^{1,p}(\Omega)$, we have
 \[
  \norm{v_h}{W^{0,1}_p} \leq  C \left (\norm{v_0}{W^{1,p}} + \norm{f}{L^p(W^{-1,p})} + h^{-1} \norm{P_h v_0 - v_0^h}{L^p(\Omega)} \right)
 \]
  for some $C > 0$ independent of $h$.
\end{cor}
\begin{proof}
 Let $\cI_h$ denote the usual finite element interpolation operator. Using the inverse estimates \cite[Corollary 1.141]{EG2004} and \cref{prop_parabolic_mfg_proj}, we have
 \begin{align*}
  \norm{v_h}{W^{0,1}_p} & \leq \norm{\cI_h v}{W^{0,1}_p} + \norm{\cI_h v - P_h v}{W^{0,1}_q} + \norm{v_h - P_h v}{W^{0,1}_p} \\
  & \leq C \left( \norm{v}{W^{0,1}_p} + h^{-1} \left( \norm{\cI_h v - P_h v}{L^q(Q_T)} + \norm{v_h - P_h v}{L^q(Q_T)} \right) \right) \\
  & \leq C \left( \norm{v}{W^{0,1}_p} + h^{-1} \left( \norm{\cI_h v - v}{L^q(Q_T)} + \norm{v_h - P_h v}{L^q(Q_T)} \right) \right) \\
  & \leq C \left( \norm{v}{W^{0,1}_p} + h^{-1} \norm{v_h - P_h v}{L^p(Q_T)} \right).
 \end{align*}
 From the error estimate in \cref{thm_mfg_parabolic_semi-discrete_heat_error_p}, we have
 \[
  \norm{v_h - P_h v}{L^p(Q_T)} \leq C \left(h \norm{v}{W^{0,1}_p} + \norm{P_h v_0 - v_0^h}{L^p(\Omega)} \right).
 \]
  The conclusion then follows from the parabolic $W^{1,p}$ estimates.
\end{proof}

From \cref{thm_mfg_parabolic_semi-discrete_heat_error_p} we deduce some further error estimates.
\begin{cor}\label{cor:quasi_optimal_q}
 Assume that \ref{h_parabolic_regularity_q} holds in addition to the assumptions of \cref{thm_mfg_parabolic_semi-discrete_heat_error_p}. Then, if $v \in W^{0,1}_q(Q_T)$, where $q$ is fixed in \ref{h_parabolic_regularity_q}, we have
 \begin{align}
  \norm{v - v_h}{W^{0,1}_q} & \leq C \left(h^{-1} \norm{P_h v_0 - v_0^h}{L^q(\Omega)} + \inf_{\tilde v_h \in  L^q(0,T;V_h)} \norm{v - \tilde v_h}{W^{0,1}_q} \right), \\
  \norm{v - v_h}{L^q(Q_T)} & \leq C \left(\norm{P_h v_0 - v_0^h}{L^q(\Omega)} + h \left[\inf_{\tilde v_h \in  L^q(0,T;V_h)} \norm{v - \tilde v_h}{W^{0,1}_q} \right] \right). \label{eq:quasi_optimal_Lq}
 \end{align}
\end{cor}

\begin{proof}
  Using the triangular inequality, we have
  \begin{equation}\label{eq:cor_quasi_optimal_q_triangle}
  \begin{split}
   \norm{v - v_h}{W^{0,1}_q} & \leq \norm{v - P_h v}{W^{0,1}_q} + \norm{P_h v - v_h}{W^{0,1}_q} \\
   & \leq \norm{v - \cI_h v}{W^{0,1}_q} + \norm{\cI_h v - P_h v}{W^{0,1}_q} + \norm{P_h v - v_h}{W^{0,1}_q},
  \end{split}
  \end{equation}
  where $\cI_h$ denotes the usual finite element interpolation operator. Since $\cI_h v = P_h \circ \cI_h v$ we obtain, using \cref{prop_parabolic_mfg_proj},
  \begin{equation}
    \norm{v - v_h}{W^{0,1}_q} \leq C \norm{v - \cI_h v}{W^{0,1}_q} + \norm{P_h v - v_h}{W^{0,1}_q}
  \end{equation}
  for some $C > 0$. From the inverse estimate \cite[Corollary 1.141]{EG2004}, we have
  \[
   \norm{P_h v - v_h}{W^{0,1}_q} \leq Ch^{-1} \norm{P_hv - v_h}{L^q(Q_T)},
  \]
  and using \cref{thm_mfg_parabolic_semi-discrete_heat_error_p}, we get
  \begin{equation}
    \norm{v - v_h}{W^{0,1}_q} \leq C \left(\norm{v - \cI_h v}{W^{0,1}_q} + h^{-1}\norm{v - R_h v}{L^q(Q_T)} + h^{-1} \norm{P_h v_0 - v_0^h}{L^q(\Omega)} \right)
  \end{equation}
  for some $C >0$. Since, as a consequence of \ref{h_parabolic_regularity_q} and using an Aubin-Nitsche duality argument, we have
  \begin{equation}\label{eq:cor_quasi_optimal_q_Rh}
   \norm{v - R_h v}{L^q(Q_T)} \leq C h\norm{v - R_h v}{W^{0,1}_q},
  \end{equation}
  it follows that
  \begin{equation}
    \norm{v - v_h}{W^{0,1}_q} \leq C \left(\norm{v - \cI_h v}{W^{0,1}_q} + \norm{v - R_h v}{W^{0,1}_q} + h^{-1} \norm{P_h v_0 - v_0^h}{L^q(\Omega)} \right).
  \end{equation}
    Notice that
    \begin{align*}
     \norm{v - \cI_h v}{W^{0,1}_q} & \leq \norm{v - \tilde v_h}{W^{0,1}_q} + \norm{\tilde v_h - \cI_h v}{W^{0,1}_q} \\
     & = \norm{v - \tilde v_h}{W^{0,1}_q} + \norm{\cI_h\left( \tilde v_h - v \right)}{W^{0,1}_q} \\
     & \leq C \norm{v - \tilde v_h}{W^{0,1}_q}
    \end{align*}
    for every $\tilde v_h \in L^q(0,T;V_h)$, where we have used the $W^{1,q}$ stability of $\cI_h$ \cite[Corollary 1.109]{EG2004}, since $q > d$. It follows that
    \[
     \norm{v - \cI_h v}{W^{0,1}_q} \leq C \inf_{\tilde v_h \in L^q(0,T;V_h)} \norm{v - \tilde v_h}{W^{0,1}_q}.
    \]
    Moreover, using \cref{prop_parabolic_mfg_ritz} and the fact that $R_h w_h = w_h$ for all $w_h \in V_h$, one may prove similarly that
    \begin{equation}\label{eq:quasi_optimal_Rh}
     \norm{v - R_h v}{W^{0,1}_q} \leq C \inf_{\tilde v_h \in L^q(0,T;V_h)} \norm{v - \tilde v_h}{W^{0,1}_q}.
    \end{equation}
    This proves the first inequality.
  For the second one, arguing as in \eqref{eq:cor_quasi_optimal_q_triangle}, we have
  \[
   \norm{v - v_h}{L^q(Q_T)} \leq C \norm{v - \cI_h v}{L^q(Q_T)} + \norm{P_h v - v_h}{L^q(Q_T)}.
  \]
  Using \cref{thm_mfg_parabolic_semi-discrete_heat_error_p}, we deduce
  \[
    \norm{v - v_h}{L^q(Q_T)} \leq  C \norm{v - \cI_h v}{L^q(Q_T)} + \norm{v - R_h v}{L^q(Q_T)} + \norm{P_h v_0 - v_0^h}{L^q(\Omega)}.
  \]
  Noticing that
  \[
   \norm{v - \cI_h v}{L^q(Q_T)} = \norm{\left(v - R_h v \right) - \cI_h \left(v - R_h v \right)}{L^q(Q_T)} \leq Ch \norm{v - R_h v}{W^{0,1}_q},
  \]
  where we have used \cite[Corollary 1.109]{EG2004} to obtain the last inequality, and using \eqref{eq:cor_quasi_optimal_q_Rh} we conclude that
  \[
   \norm{v - v_h}{L^q(Q_T)} \leq  C \left( \norm{P_h v_0 - v_0^h}{L^q(\Omega)} + h \norm{v - R_h v}{W^{0,1}_q}\right).
  \]
  The estimate \eqref{eq:quasi_optimal_Lq} then follows from \eqref{eq:quasi_optimal_Rh}.
\end{proof}

In what follows we denote by $S_I^h$ and $S_{IS}^h$ the linear operators defined by $S_I^h v_0^h = v_h$, where $v_h$ solves \eqref{eq_mfg_parabolic_smooth_semi-discrete} with $f=0$, and $S^h_{IS} f =  u_h$, where $u_h$ solves \eqref{eq_mfg_parabolic_smooth_semi-discrete} with $v_0^h = 0$. We also consider $S^h_T$ and $S_{TS}$ their backward analogues, i.e.,
\[
 \left(S^h_T v_0^h \right)(t,x) = \left(S^h_I v_0^h \right)(T-t,x) \quad  \textnormal{and} \quad \left(S^h_{TS} f \right)(t,x) = \left(S^h_{IS} f \right)(T - t, x).
\]
Of course, these operators must be seen as semi-discrete finite-element approximations of $S_I$, $S_{IS}$, $S_T$ and $S_{TS}$ defined in \eqref{eq_mfg_parabolic_smooth_SI}, \eqref{eq_mfg_parabolic_smooth_SIS}, \eqref{eq_mfg_parabolic_smooth_ST} and \eqref{eq_mfg_parabolic_smooth_STS}, respectively.

\begin{lem}\label{lem_mfg_parabolic_errors}
 Assume \ref{h_parabolic_regularity_q}. Then
 \begin{gather*}
  S_I^h \circ P_h \in \cL(L^\infty(\Omega), L^q(Q_T)),\, S_T^h  \circ P_h \in \cL(W^{2,\infty}(\Omega), W^{0,1}_q(Q_T)), \\ \, S_{IS}^h \in \cL(L^{q/2}(0,T;W^{-1,q/2}(\Omega)), L^q(Q_T)), \, S_{TS}^h \in \cL(L^{q/2}(Q_T), W^{0,1}_q(Q_T)),
 \end{gather*}
 where $q$ is fixed in \ref{h_parabolic_regularity_q}, and we have
 \begin{align*}
  \norm{S_I - S_I^h \circ P_h}{\cL(L^\infty,L^q)} +  \norm{S_{IS} - S_{IS}^h}{\cL(L^{q/2}(W^{-1,q/2}),L^q)}  \leq C h^{2/q} \left (1 + \module{\ln(h)}^{1 - 2/q} \right), \\
  \norm{S_T - S_T^h \circ P_h}{\cL(W^{2,\infty}, W^{0,1}_q)} + \norm{S_{TS} - S_{TS}^h}{\cL(L^{q/2}, W^{0,1}_q)} \leq C\left (h + h^{1 - d/q} + h^{1/2}\module{\ln(h)}^{1/2} \right).
 \end{align*}
\end{lem}
\begin{proof}\, We consider each operator separately. \\
\noindent\underline{\textit{Estimate on $S_I^h$.}}
 First, using the triangle inequality, we have
 \[
  \norm{S_I - S_I^h \circ P_h}{\cL(L^\infty,L^q)} \leq \norm{S_I - P_h \circ S_I}{\cL(L^\infty,L^q)} + \norm{P_h \circ S_I - S_I^h \circ P_h}{\cL(L^\infty,L^q)}.
 \]
 From \cref{thm_mfg_parabolic_semi-discrete_heat_error_p}, for every $v_0 \in L^\infty$, we have
 \[
  \norm{P_h \circ S_I v_0 - S_I^h \circ P_h v_0}{L^q(Q_T)} \leq C \norm{S_I v_0 - R_h \circ S_I v_0}{L^q(Q_T)}.
 \]
 Using \cref{prop_parabolic_mfg_ritz}, we have that
 \[
  \norm{S_I v_0 - R_h \circ S_I v_0}{L^\infty(Q_T)} \leq C \module{\ln(h)} \norm{S_I v_0}{L^\infty(Q_T)} \leq C \module{\ln(h)} \norm{v_0}{L^\infty(\Omega)}.
 \]
 On the other hand, also using \cref{prop_parabolic_mfg_ritz}, we have
 \[
  \norm{S_I v_0 - R_h \circ S_I v_0}{L^2(Q_T)} \leq Ch \norm{S_I v_0}{W^{0,1}_2} \leq C h \norm{v_0}{L^\infty(\Omega)}.
 \]
 By interpolation, we deduce that
 \[
   \norm{S_I v_0 - R_h \circ S_I v_0}{L^q(Q_T)} \leq C h^{2/q} \module{\ln(h)}^{1 - 2/q} \norm{v_0}{L^\infty(\Omega)}.
 \]
Moreover, from \cref{prop_parabolic_FP_wellposed,prop_parabolic_FP_DGNM}, we have $S_I \in \cL(L^\infty(\Omega),\cH_2^1(Q_T) \cap L^\infty(Q_T))$. We deduce from \cref{prop_parabolic_mfg_proj} that
 \begin{align*}
  \norm{S_I v_0 - P_h \circ S_I v_0}{L^\infty(Q_T)} & \leq C \norm{S_I v_0}{L^\infty(Q_T)} \leq C \norm{v_0}{L^\infty(\Omega)}, \\
  \norm{S_I v_0 - P_h \circ S_I v_0}{L^2(Q_T)} & \leq Ch \norm{S_I v_0}{W^{0,1}_2} \leq Ch \norm{v_0}{L^\infty(\Omega)}.
 \end{align*}
 Then, by interpolation, we obtain
 \[
  \norm{S_I v_0 - P_h \circ S_I v_0}{L^q(Q_T)} \leq Ch^{2/q} \norm{v_0}{L^\infty(\Omega)}.
 \]
 We conclude that
 \[
  \norm{S_I v_0 - S_I^h \circ P_h v_0}{L^q(Q_T)} \leq Ch^{2/q} \left (1 + \module{\ln(h)}^{1 - 2/q} \right) \norm{v_0}{L^\infty(\Omega)},
 \]
 and hence that
 \[
  \norm{S_I - S_I^h \circ P_h}{\cL(L^\infty,L^q)} \leq Ch^{2/q} \left (1 + \module{\ln(h)}^{1 - 2/q} \right).
 \]

 \noindent\underline{\textit{Estimate on $S_T^h$.}} Let $v_T \in W^{2,\infty}(\Omega)$. Using the triangle inequality, we have
 \[
  \norm{S_T v_T - S_T^h \circ P_h v_T}{W^{0,1}_q} \leq \norm{S_T v_T - P_h \circ S_T v_T}{W^{0,1}_q} + \norm{P_h \circ S_T v_T - S_T^h \circ P_h v_T}{W^{0,1}_q}.
 \]
 Moreover, from \cref{thm_mfg_parabolic_semi-discrete_heat_error_p}, \cref{prop_parabolic_mfg_ritz} and the inverse estimates \cite[Theorem 4.5.11]{BS2008}, we have
 \begin{align*}
  \norm{P_h \circ S_T v_T - S_T^h \circ P_h v_T}{W^{0,1}_q} & \leq Ch^{-1} \norm{P_h \circ S_T v_T - S_T^h \circ P_h v_T}{L^q(Q_T)} \\
  & \leq Ch^{-1} \norm{S_T v_T - R_h \circ S_T v_T}{L^q(Q_T)}.
 \end{align*}
 As a consequence of \cref{prop_parabolic_mfg_ritz}, we have
 \[
  \norm{S_T v_T - R_h \circ S_T v_T}{L^\infty(Q_T)} \leq Ch \module{\ln(h)} \norm{S_T v_T}{L^\infty(W^{1,\infty})} \leq Ch \module{\ln(h)} \norm{v_T}{W^{2,\infty}}.
 \]
 On the other hand, since we assume \ref{h_parabolic_regularity_q}, we have
 \begin{align*}
  \norm{S_T v_T - R_h \circ S_T v_T}{L^{q/2}(Q_T)} \leq C h^2 \norm{S_T v_T}{W^{1,2}_{q/2}} \leq C h^2 \norm{v_T}{W^{2,\infty}}.
 \end{align*}
 By interpolation we deduce
 \[
  \norm{S_T v_T - R_h \circ S_T v_T}{L^{q}(Q_T)} \leq C h^{3/2} \module{\ln(h)}^{1/2} \norm{v_T}{W^{2,\infty}}
 \]
 and it follows that
 \[
  \norm{P_h \circ S_T v_T - S_T^h \circ P_h v_T}{W^{0,1}_q} \leq C h^{1/2} \module{\ln(h)}^{1/2} \norm{v_T}{W^{2,\infty}}.
 \]
 In addition, using \cite[Corollary 1.109]{EG2004} and the inverse estimates \cite[Theorem 4.5.11]{BS2008},
 \begin{align*}
  \norm{S_T v_T - P_h \circ S_T v_T}{W^{0,1}_q} & \leq \norm{S_T v_T - \cI_h \circ S_T v_T}{W^{0,1}_q} + \norm{\cI_h \circ S_T v_T - P_h \circ S_T v_T}{W^{0,1}_q} \\
  & \leq C h \norm{v_T}{W^{2,\infty}} + \norm{\cI_h \circ S_T v_T - P_h \circ S_T v_T}{W^{0,1}_q} \\
  & \leq C \left (h \norm{v_T}{W^{2,\infty}} + h^{-(1 + d/q)} \norm{\cI_h \circ S_T v_T - P_h \circ S_T v_T}{L^{q/2}(Q_T)} \right)
 \end{align*}
 Using the fact that $P_h = I$ on $V_h$ and \cref{prop_parabolic_mfg_proj}, we also have
 \begin{align*}
  \norm{\cI_h \circ S_T v_T - P_h \circ S_T v_T}{L^{q/2}(Q_T)} & \leq C \norm{\cI_h \circ S_T v_T - S_T v_T}{L^{q/2}(Q_T)} \\
  & \leq C h^2 \norm{S_T v_T}{W^{1,2}_{q/2}} \\
  & \leq Ch^2 \norm{v_T}{W^{2,\infty}}.
 \end{align*}
 It follows that
 \[
  \norm{S_T v_T - P_h \circ S_T v_T}{W^{0,1}_q} \leq C \left ( h + h^{1 - d/q} \right) \norm{v_T}{W^{2,\infty}},
 \]
 and we conclude that
 \[
  \norm{S_T - S_T^h \circ P_h}{\cL(W^{2,\infty},W^{0,1}_q)} \leq C \left(h + h^{1 - d/q} + h^{1/2}\module{\ln(h)}^{1/2} \right).
 \]

\noindent\underline{\textit{Estimate on $S_{IS}^h$.}} Let $g \in L^{q/2}(0,T;W^{-1,q/2}(\Omega))$ and notice that
\[
 \norm{S_{IS} g - S_{IS}^h g}{L^q(Q_T)} \leq \norm{S_{IS} g - P_h \circ S_{IS} g}{L^q(Q_T)} + \norm{P_h \circ S_{IS} g - S_{IS}^h g}{L^q(Q_T)}.
\]
From \cref{thm_mfg_parabolic_semi-discrete_heat_error_p}, we have
\[
 \norm{P_h \circ S_{IS} g - S_{IS}^h g}{L^q(Q_T)} \leq C \norm{S_{IS} g - R_h \circ S_{IS} g}{L^q(Q_T)}.
\]
On the one hand, since \cref{prop_parabolic_FP_DGNM} implies that $S_{IS} g \in L^\infty(Q_T)$, we deduce from \cref{prop_parabolic_mfg_ritz} that
\[
 \norm{S_{IS} g - R_h \circ S_{IS} g}{L^\infty(Q_T)} \leq C \module{\ln(h)} \norm{S_{IS}g}{L^\infty(Q_T)} \leq C \module{\ln(h)} \norm{g}{L^{q/2}(W^{-1,q/2})}.
\]
On the other hand, using \cref{thm_mfg_parabolic_semi-discrete_heat_error_2}, we have
\[
 \norm{S_{IS} g - R_h \circ S_{IS} g}{L^2(Q_T)} \leq Ch \norm{g}{L^{q/2}(W^{-1,q/2})}.
\]
By interpolation we conclude that
\[
 \norm{S_{IS} g - R_h \circ S_{IS} g}{L^q(Q_T)} \leq Ch^{2/q}\module{\ln(h)}^{1 - 2/q} \norm{g}{L^{q/2}(W^{-1,q/2})}.
\]
Moreover, an argument similar to the one used in the case of $S_I^h$ yields the estimate
\[
 \norm{S_{IS} g - P_h \circ S_{IS} g}{L^q(Q_T)} \leq Ch^{2/q}\norm{g}{L^{q/2}(W^{-1,q/2})}.
\]
It follows that
\[
 \norm{S_{IS} - S_{IS}^h}{\cL(L^{q/2}(W^{-1,q/2}),L^q)} \leq C h^{2/q} \left ( 1 + \module{\ln(h)}^{1 - 2/q} \right).
\]

\noindent\underline{\textit{Estimate on $S_{TS}^h$.}} Let $f \in L^{q/2}(Q_T)$ and write
\[
 \norm{S_{TS} f - S_{TS}^h f}{W^{0,1}_q} \leq \norm{S_{TS} f - P_h \circ S_{TS} f}{W^{0,1}_q} + \norm{P_h \circ S_{TS} f - S_{TS}^h f}{W^{0,1}_q}.
\]
Using the inverse estimates \cite[Theorem 4.5.11]{BS2008} and \cref{thm_mfg_parabolic_semi-discrete_heat_error_p}, we have
\begin{align*}
 \norm{P_h \circ S_{TS} f - S_{TS}^h f}{W^{0,1}_q} & \leq Ch^{-1} \norm{P_h \circ S_{TS} f - S_{TS}^h f}{L^q(Q_T)} \\
 & \leq C h^{-1} \norm{S_{TS} f - R_h \circ S_{TS} f}{L^q(Q_T)}.
\end{align*}
Arguing similarly to the case of $S_{T}^h$, we obtain
\[
 \norm{S_{TS} f - R_h \circ S_{TS} f}{L^q(Q_T)} \leq C h^{3/2} \module{\ln(h)}^{1/2}\norm{f}{L^{q/2}(Q_T)}
\]
and
\[
 \norm{S_{TS} f - P_h \circ S_{TS} f}{W^{0,1}_q} \leq C \left (h + h^{1 - d/q} \right) \norm{f}{L^{q/2}(Q_T)},
\]
so that
\[
 \norm{S_{TS} - S_{TS}^h}{\cL(L^{q/2}, W^{0,1}_q)} \leq C\left( h + h^{1 - d/q} + h^{1/2}\module{\ln(h)}^{1/2} \right).
\]
\end{proof}

\subsection{Semidiscrete approximations of stable solutions to the MFG system}

 We consider the Banach spaces $X$, $Y_1$ and $Y_2$ defined in \eqref{eq_mfg_parabolic_smooth_spaces}. Recalling that $(u,m) \in X$ is a weak solution to \eqref{eq_mfg_parabolic_smooth} if and only if $\Upsilon(u,m) = 0$, where $\Upsilon \colon X \to X$ is defined in \eqref{eq_mfg_parabolic_smooth_def_Upsilon}. A solution to the semi-discrete MFG system \eqref{eq:mfg_semi-discrete_HJB}-\eqref{eq:mfg_semi-discrete_FP} is then a pair $(u_h,m_h) \in X$ such that $\Upsilon_h(u_h,m_h) = 0$, where
\begin{align}
 \Upsilon_h(v,\rho) & := \left (I - \cS_1^h \cR_1^h - \cS_2^h \cR_2 \right)(v,\rho), \\
 \cR_1^h(v,\rho) & := \begin{pmatrix} u_T^h \\ m_0^h \end{pmatrix}, \label{eq:R_1_h}
\end{align}
for all $(v,\rho) \in X$, and
\[
 \cS_1^h := \begin{pmatrix}S_T^h & 0 \\ 0 & S_I^h \end{pmatrix}, \quad \cS_2^h := \begin{pmatrix} S_{TS}^h &  \\ 0 & S_{IS}^h \end{pmatrix}.
\]

In order to prove our main result, we are going to apply the following version of the Brezzi-Rappaz-Raviart approximation theorem \cite{BRR1980}.
\begin{thm}[Brezzi-Rappaz-Raviart, {\cite[Theorem 3.2]{BLS2025a}}]\label{thm:reworked_BRR}
 Let $X$ and $Y$ be Banach spaces, let $F \colon X \to Y$ be continuous and let $\bar x \in X$ be such that $F(\bar x) = 0$. For every $h > 0$, let $F_h \colon X \to Y$ be continuous and assume that
\begin{enumerate}[label={\rm{(\roman*)}}]
 \item $F_h(\bar x) \xrightarrow{h \to 0} 0$;
 \item $F$ and $F_h$ are Fréchet differentiable at $\bar x$, with
 \[
  \lim_{h \to 0} \norm{dF[\bar x] - dF_h[\bar x]}{\cL(X,Y)} = 0,
 \]
 and there exists $c \colon \RR_+ \to \RR_+$, nondecreasing with $c(0^+) = 0$, such that, for all $x,\, x' \in B_X(\bar x,r)$ and $r > 0$, we have
 \begin{align*}
  \norm{F(x) - F(x') - dF[\bar x](x - x')}{Y} & \leq c(r) \norm{x - x'}{X}, \\
  \norm{F_h(x) - F_h(x') - dF_h[\bar x](x - x')}{Y} & \leq c(r) \norm{x - x'}{X};
 \end{align*}
 \item $dF[\bar x]$ is an isomorphism.
\end{enumerate}
 Then there exists a neighborhood $\mathcal{O}$ of $\bar x$ and $h_0 > 0$ such that, for all $0 < h \leq h_0$, there exists a unique $\bar x_h \in \mathcal{O}$ such that $F_h(\bar x_h) = 0$
 and we have the error estimate
 \[
  \norm{\bar x - \bar x_h}{X} \leq 2 \norm{dF[\bar x]^{-1}}{\cL(Y,X)} \norm{F_h(\bar x)}{Y}.
 \]
\end{thm}

We are now able to state and prove our main result.
\begin{thm}\label{thm:error_estimate_parabolic}
   Assume \ref{h_parabolic_mfg}, \ref{h_parabolic_regularity_q} and \ref{h_parabolic_mfg_diff} and let $q > 2(d+2)$ as in \ref{h_parabolic_regularity_q}. Let $(u,m) \in X$ be a stable solution to \eqref{eq_mfg_parabolic_smooth} and assume that
    \begin{equation}\label{eq:thm_error_estimate_parabolic_1}
     \lim_{h \to 0} \norm{m_0 - m_0^h}{L^q(\Omega)} + \norm{u_T - u_T^h}{W^{1,q}} = 0.
    \end{equation}
    Then there exist $h_0 > 0$ and a neighborhood $\mathcal{O}$ of $(u,m)$ in $X$, such that, for every $0 < h \leq h_0$, there exists $(u_h,m_h) \in \mathcal{O}$ such that $\Upsilon_h(u_h,m_h) = 0$. Moreover, $(u_h,m_h)$ is the unique zero of $\Upsilon_h$ in $\mathcal{O}$ and we have the error estimate
   \begin{multline}\label{eq:error_estimate_parabolic_simple}
    \norm{(u-u_h, m-m_h)}{X} \\ \leq C \left( \norm{P_h m_0 - m_0^h}{L^q(\Omega)} + \norm{P_h u_T - u_T^h}{W^{1,q}} + h^{2/q} \left( 1 + \module{\ln(h)}^{1-2/q} \right) \right).
   \end{multline}
   Furthermore, if $m \in W^{0,1}_q(Q_T)$, then we have
   \begin{multline}\label{eq:error_estimate_parabolic_quasi_optimal}
    \norm{u - u_h}{W^{0,1}_q} + \norm{m - m_h}{L^q(Q_T)} \\ \leq C \left( \norm{P_h m_0 - m_0^h}{L^q(\Omega)} + h^{-1}\norm{P_h u_T - u_T^h}{L^q(\Omega)} \right) \\ + C \left[ \inf_{(v_h,\rho_h) \in L^q(0,T;V_h)^2} \norm{u - v_h}{W^{0,1}_q} + h \norm{m - \rho_h}{W^{0,1}_q} \right].
   \end{multline}
\end{thm}
\begin{proof}[Proof of \cref{thm:error_estimate_parabolic}]
 First, notice that
  \begin{multline}\label{eq:thm_error_estimate_parabolic_2}
   \norm{\cS_1 \cR_1(u,m) - \cS_1^h \cR_1^h(u,m)}{X}
   \\ \leq \norm{S_I^h \circ P_h \left(m_0 - m_0^h \right)}{L^q(Q_T)} + \norm{S_T^h \circ P_h \left(u_T - u_T^h \right)}{W^{0,1}_q} \\ + \norm{\left(S_I - S_I^h \circ P_h \right) m_0 }{L^q(Q_T)} + \norm{\left(S_T - S_T^h \circ P_h \right) u_T}{W^{0,1}_q}.
  \end{multline}
  From \cref{thm_mfg_parabolic_semi-discrete_heat_error_p}, \cref{cor:W_0_1_p_stability} and \cref{prop_parabolic_mfg_proj} we have
  \begin{multline}
   \norm{S_I^h \circ P_h \left(m_0 - m_0^h \right)}{L^q(Q_T)} + \norm{S_T^h \circ P_h \left(u_T - u_T^h \right)}{W^{0,1}_q} \\ \leq C \left(\norm{m_0 - m_0^h}{L^q(\Omega)} + \norm{u_T - u_T^h}{W^{1,q}} \right),
  \end{multline}
  so that \eqref{eq:thm_error_estimate_parabolic_1} yields
  \[
   \lim_{h \to 0 } \norm{S_I^h \circ P_h \left(m_0 - m_0^h \right)}{L^q(Q_T)} + \norm{S_T^h \circ P_h \left(u_T - u_T^h \right)}{W^{0,1}_q} = 0.
  \]
 We then conclude from \cref{lem_mfg_parabolic_errors} that
 \[
  \lim_{h \to 0} \Upsilon_h(u,m) = \Upsilon(u,m) = 0.
 \]
 From \cref{thm_mfg_parabolic_isom}, we know that $d\Upsilon[u,m]$ is an isomorphism on $X$. Moreover, using \cref{prop_parabolic_mfg_diff}, we know that $\Upsilon_h$ is strictly differentiable at $(u,m)$ with
\[
 d\Upsilon_h[u,m] = I - \cS_2^h \circ d\cR_2[u,m],
\]
and since $\left(\cS_2^h\right)_{h > 0}$ is bounded in operator norm, there exists a nondecreasing function $c \colon \RR_+^* \to \RR_+$, satisfying $c(0^+) = 0$, such that
\[
 \norm{\Upsilon_h(v_1,\rho_1) - \Upsilon_h(v_2,\rho_2) - d\Upsilon_h[u,m](v_1 - v_2, \rho_1 - \rho_2)}{X} \leq c(r) \norm{(v_1 - v_2, \rho_1 - \rho_2)}{X}
\]
for all $(v_1,\rho_1), \, (v_2,\rho_2) \in B_X((u,m),r)$ and $h > 0$.
In addition, it follows from \cref{lem_mfg_parabolic_errors} we have
 \begin{align*}
  \norm{d\Upsilon[u,m] - d \Upsilon_h[u,m]}{\cL(X)}
  & \leq \norm{\cS_2 - \cS_2^h}{\cL(Y_2,X)} \norm{d\cR_2[u,m]}{\cL(X,Y_2)} \\
  & \leq C h^{\kappa} \left (1 + \module{\ln(h)}^{1 - \kappa} \right) \norm{d\cR_2[u,m]}{\cL(X,Y_2)}
 \end{align*}
 for some $\kappa \in (0,1)$, so that $\lim_{h \to 0} \norm{d\Upsilon[u,m] - d \Upsilon_h[u,m]}{\cL(X)} = 0$.
 We can therefore apply the Brezzi-Rappaz-Raviart approximation theorem \ref{thm:reworked_BRR} to deduce the existence of $(u_h,m_h)$ and the error estimate
 \[
  \norm{(u-u_h,m-m_h)}{X} \leq 2 \norm{d\Upsilon[u,m]^{-1}}{\cL(X)} \norm{\Upsilon_h(u,m)}{X}.
 \]
 Rewriting
 \begin{align*}
  \norm{\Upsilon_h(u,m)}{X} & = \norm{\Upsilon_h(u,m) - \Upsilon(u,m)}{X} \\ & \leq \norm{\cS_1 \cR_1(u,m) - \cS_1^h \cR_1^h(u,m)}{X} + \norm{\left(\cS_2^h - \cS_2\right) \cR_2(u,m)}{X},
 \end{align*}
 we deduce \eqref{eq:error_estimate_parabolic_simple} from \eqref{eq:thm_error_estimate_parabolic_2} and \cref{lem_mfg_parabolic_errors}. Then, estimate \eqref{eq:error_estimate_parabolic_quasi_optimal} is a consequence of \cref{cor:quasi_optimal_q}.
\end{proof}

 \begin{rem}
 Let us comment on the role of assumption \ref{h_parabolic_regularity_q} and the choice of the value of $q$. We want to apply the Brezzi-Rappaz-Raviart theorem in the space $X = W^{0,1}_q(Q_T) \times L^q(Q_T)$ for some $q \geq 2$ to be determined. We need the invertibility of the differential $d \Upsilon[u,m]$ which is obtained in \cref{thm_mfg_parabolic_isom} by using Fredholm's alternative, using the compactness of the linear operators $S_{IS}$ and $S_{TS}$. The compactness of $S_{IS}$ follows from the De Giorgi-Nash-Moser estimates \cref{prop_parabolic_FP_DGNM} and requires that $\rho H_p(Du) + m H_{pp}(Du)Dv$
 belongs to $L^r(Q_T)$ for every $(v,\rho) \in X$, with $r > d+2$. It is then natural to choose $q = 2r$ in order to apply Hölder's inequality, given the growth conditions in \ref{h_parabolic_mfg}, which yields $q > 2(d+2)$. On the other hand, in order to obtain the compactness of $S_{TS}$, we assume that the $L^p$-shift property holds for some $p > 1$. Since the growth conditions \ref{h_parabolic_mfg} on the Hamiltonian give $H(Du) \in L^{q/2}(Q_T)$, we need the shift property to hold for $p = q/2$. Finally we need to obtain a compact embedding $W^{1,2}_{q/2}(Q_T) \hookrightarrow W^{0,1}_q(Q_T)$, which holds for $q > \max \{ d+3, 2d \}$ according to \cref{lem:parabolic_compact_embedding}. In the end, the strongest restriction on $q$ is $q > 2(d+2)$.
\end{rem}

\begin{cor}\label{cor:quasi_optimal}
  Under the assumptions of \cref{thm:error_estimate_parabolic} and if $m \in W^{0,1}_q(Q_T) \cap L^\infty(Q_T)$ and
  \begin{equation}\label{eq:quasi_optimal_assumption}
   \sup_{h \in (0,h_0]} \norm{H_p(\cdot,Du_h)}{L^\infty} \leq K
  \end{equation}
  for some $K >0$, then
   \begin{multline}\label{eq:error_estimate_parabolic_true_quasi_optimal}
    \norm{u - u_h}{W^{0,1}_q} + \norm{m - m_h}{W^{0,1}_q} \\ \leq C h^{-1}\left(\norm{P_h m_0 - m_0^h}{L^q(\Omega)} + \norm{P_h u_T - u_T^h}{L^q(\Omega)} \right) \\
    + C \left[ \inf_{(v_h,\rho_h) \in L^q(0,T;V_h)^2} \norm{u - v_h}{W^{0,1}_q} + \norm{m - \rho_h}{W^{0,1}_q} \right].
   \end{multline}
\end{cor}

\begin{rem}
 Note that \eqref{eq:quasi_optimal_assumption} holds for instance if we also assume that $H$ is Lipschitz continuous in $p$, uniformly in $x$.
\end{rem}

\begin{proof}[Proof of \cref{cor:quasi_optimal}]
 Using \cite[Proposition 66.2]{EG2021}, let $\widetilde m_h \in H^1(0,T;V_h)$ be the unique solution to
\begin{multline}\label{eq:tilde_m_h}
 \int_0^t \int_{\Omega} \partial_t \widetilde m_h(s,x) \phi_h(s,x) +  D \widetilde m_h(s,x) \cdot D \phi_h(s,x) \, dx ds \\ = - \int_0^t \int_\Omega m(s,x) H_p(x,Du) \cdot D \phi_h(s,x) \, dx ds
\end{multline}
for all $t \in (0,T)$ and $\phi_h \in L^2(0,T;V_h)$ satisfying $\widetilde m_h(0,\cdot) = m_0^h$. Since $m$ is the unique weak solution to
\begin{equation}
 \begin{cases}
  \partial_t \mu - \Delta \mu = \diver \left(m H_p(x,Du) \right) \quad & \textnormal{in $Q_T$,} \\
  \mu(t,\cdot) = 0 \quad & \textnormal{on $\partial \Omega$}, \\
  \mu(0,\cdot) = m_0,
 \end{cases}
\end{equation}
and we assume that $m \in W^{0,1}_q(Q_T)$, it follows from \cref{cor:quasi_optimal_q} that
\begin{equation}
  \norm{m - \widetilde m_h}{W^{0,1}_q} \leq C \left( h^{-1} \norm{P_h m_0 - m_0^h}{L^q} +  \inf_{\rho_h \in L^q(0,T;V_h)} \norm{m - \rho_h}{W^{0,1}_q} \right).
\end{equation}
Moreover, using the triangle inequality we have
\begin{equation}\label{eq:quasi_optimal_2}
 \begin{split}
 & \norm{m - m_h}{W^{0,1}_q} \leq \norm{m - \widetilde m_h}{W^{0,1}_q} + \norm{m_h - \widetilde m_h}{W^{0,1}_q} \\ & \qquad \leq \norm{m_h - \widetilde m_h}{W^{0,1}_q} +  C \left( h^{-1} \norm{P_h m_0 - m_0^h}{L^q} + \inf_{\rho_h \in L^q(0,T;V_h)} \norm{m - \rho_h}{W^{0,1}_q} \right).
 \end{split}
\end{equation}
In order to estimate $\norm{m_h - \widetilde m_h}{W^{0,1}_q}$, we notice that $\mu_h := m_h - \widetilde m_h$ solves
\begin{multline*}
 \int_0^t \int_{\Omega} \partial_t \mu_h(s,x) \phi_h(s,x) +  D \mu_h(s,x) \cdot D \phi_h(s,x) \, dx ds \\ = - \int_0^t \int_\Omega (m_h(s,x) - m(s,x)) H_p(x,Du_h) \cdot D \phi_h(s,x) \, dx ds\\
 -  \int_0^t \int_\Omega m(s,x)(H_p(x,Du_h) - H_p(x,Du)) \cdot D \phi_h(s,x) \, dx ds
\end{multline*}
for all $t \in (0,T)$ and $\phi_h \in L^2(0,T;V_h)$, with $\mu_h(0,\cdot) = 0$. Since $\Omega$ is convex and \eqref{eq:quasi_optimal_assumption} holds, we can apply \cref{cor:W_0_1_p_stability} to deduce that
\begin{equation}\label{eq:quasi_optimal_1}
 \norm{m_h - \widetilde m_h}{W^{0,1}_q} \leq C \left( \norm{m - m_h}{L^q} + \norm{Du - Du_h}{L^q} \right).
\end{equation}
We then obtain \eqref{eq:error_estimate_parabolic_true_quasi_optimal} by combining \eqref{eq:error_estimate_parabolic_quasi_optimal}, \eqref{eq:quasi_optimal_2} and \eqref{eq:quasi_optimal_1}.
\end{proof}

\subsection{Stabilized scheme}
\label{section:stabilized_scheme}

  The semi-discrete scheme \eqref{eq:mfg_semi-discrete_HJB}-\eqref{eq:mfg_semi-discrete_FP} may fail to preserve the nonnegativity of the initial distribution $m_0$. In order to solve this issue, we present a stabilized version of the scheme, following ideas from \cite{OS2024a,OS2025}.

  To that end, we are going to strengthen \ref{h_parabolic_regularity_q} so that the $L^p$-shift property holds up to $p = q$ instead of $q/2$.
  \begin{enumerate}[label={\bf\textrm{(H2+)}}]
   \item \label{h_parabolic_regularity_q_strong} The domain $\Omega$ is such that there exists $q > 2(d+2)$ such that the \emph{$L^p$-shift property} holds for every $1 < p \leq q$, \textit{i.e.}, there exists $C_p > 0$ such that
 \[
    \norm{v}{W^{2,p}} \leq C_p \left(\norm{f}{L^p(\Omega)} + \norm{v}{L^p(\Omega)} \right)
 \]
 whenever $v \in W^{1,p}_0(\Omega)$ is a weak solution to $\Delta v = f$ with homogeneous Dirichlet boundary conditions.
  \end{enumerate}

  \subsubsection{The method of lumped masses}

  In order to obtain a semi-discrete scheme satisfying a maximum principle, we are going to leverage the method of lumped masses \cite[Chapter 15]{T2006}. More precisely, we define the following inner-product on $V_h$
  \[
    \langle \phi_h,\psi_h \rangle_{h} = \int_\Omega \cI_h \left[ \phi_h \psi_h \right](x) \, dx = \sum_{k=1}^{N_h} \phi_h(x_h^k) \psi_h(x_h^k) \langle \xi^k_h\rangle \quad \textnormal{for all } \phi_h, \, \psi_h \in V_h,
  \]
  where $(x^k_h)_{1 \leq k \leq N_h}$ are the vertices of the triangulation $\cT_h$ corresponding to $V_h$, $(\xi_h^k)_{1 \leq k \leq N_h}$ is the associated standard $\PP^1$-Lagrange basis and $\langle f \rangle = \int_\Omega f(x)\, dx$. More generally, for $1 \leq p < \infty$, we consider the norms
  \[
    \norm{\phi_h}{h,p} := \left(\int_\Omega \cI_h \left[ \module{\phi_h}^p \right](x)\, dx \right)^{1/p} \quad \textnormal{for all } \phi_h \in V_h
  \]
  and we set
  \[
   V_h^p := (V_h, \norm{\cdot}{h,p}).
  \]
  It is known that, for every $1 \leq p < \infty$, there exists $C_p > 0$, independent of $h$, such that
  \begin{equation}\label{eq:ML_norm}
   C_p^{-1} \norm{v_h}{L^p(\Omega)} \leq \norm{v_h}{h,p} \leq C_p \norm{v_h}{L^p(\Omega)}.
  \end{equation}

  For $u_h \in V_h$, we may define $\mathbf{u}_h \in \RR^{N_h}$ by $\mathbf{u}_h = (u_h(x_h^1), \dots, u_h(x_h^{N_h}))^\top$ so that
  \[
    \langle u_h, v_h \rangle_{L^2(\Omega)} = \langle M_h \mathbf{u}_h, \mathbf{v}_h \rangle \quad \textnormal{for all } u_h, \, v_h \in V_h,
  \]
  where $(M_h)_{i,j} = \langle \xi_h^i, \xi_h^j \rangle_{L^2(\Omega)}$ is the mass matrix. Note that $M_h$ is a symmetric positive definite matrix. In particular, it is invertible. Similarly, we have
  \[
    \langle u_h, v_h \rangle_{h} = \langle M_h^L \mathbf{u}_h, \mathbf{v}_h \rangle  \quad \textnormal{for all } u_h, \, v_h \in V_h,
  \]
  with $M_h^L = \textnormal{diag}(\langle \xi_h^1\rangle, \dots, \langle \xi_h^{N_h}\rangle)$. In what follows it will be convenient to consider the linear operator $\cK_h \colon V_h \to V_h$ defined by
  \begin{equation}
    \cK_h u_h(x) =  \left \langle M_h^{-1} M_h^L \mathbf{u}_h, (\xi_h^1(x), \dots, \xi_h^{N_h}(x))^\top \right \rangle \quad \textnormal{for all } u_h \in V_h.
  \end{equation}
 Indeed, with this definition, we have
 \[
    \langle \cK_h u_h, v_h \rangle_{L^2(\Omega)} = \langle M_h M_h^{-1} M_h^L \mathbf{u}_h, \mathbf{v}_h \rangle = \langle u_h, v_h \rangle_h \quad \textnormal{for all } u_h, \, v_h \in V_h.
 \]
 Note that $\cK_h$ is invertible with
 \[
    \cK_h^{-1} u_h(x) = \left \langle \left(M_h^L \right)^{-1} M_h \mathbf{u}_h , (\xi_h^1(x), \dots, \xi_h^{N_h}(x))^\top  \right \rangle \quad \textnormal{for all } u_h \in V_h.
 \]
 The operator $\cK_h$ may be extended to $V_h'$ by setting
 \[
    \langle \cK_h f, v_h \rangle_{V_h', V_h} := \langle  f, \cK_h v_h \rangle_{V_h', V_h}, \quad \textnormal{for all } f \in V_h'.
 \]
In particular, for $f \in V_h$, we recover
\[
  \langle \cK_h f, v_h \rangle_{V_h', V_h} = \langle f, \cK_h v_h \rangle_{L^2(\Omega)} = \langle f, v_h \rangle_h.
\]
Since we assume that the triangulations $(\cT_h)_{h > 0}$ are quasi-uniform, for every $1 < p < \infty$, there exists $C_p > 0$, independent of $h$, such that
\begin{equation}\label{eq:K_h_equiv}
 C_p^{-1}\norm{v_h}{L^p(\Omega)} \leq \norm{\cK_h v_h}{L^p(\Omega)} \leq C_p \norm{v_h}{L^p(\Omega)}
\end{equation}
for every $v_h \in V_h$, see \cite{FSS2001}.

We have the following estimate.
\begin{lem}\label{lem:ML_quadrature_error}
 For every $1 < p < \infty$ there exists $C_p > 0$, independent of $h$, such that
 \[
    \norm{v_h - \cK_h^{-1} v_h}{h,p} \leq C_ph \norm{Dv_h}{L^p(\Omega)} \quad \textnormal{for all } v_h \in V_h.
 \]
\end{lem}
\begin{proof}
 Let $\phi_h \in V_h$. By an adaptation of \cite[Lemma 15.1]{T2006}, we have
 \[
    \module{\langle (v_h - \cK_h^{-1} v_h), \phi_h \rangle_h} = \module{\langle v_h, \phi_h \rangle_{h} - \langle v_h, \phi_h \rangle_{L^2(\Omega)}} \leq C_p h^2 \norm{D v_h}{L^p(\Omega)} \norm{D \phi_h}{L^{p'}(\Omega)},
 \]
 where $1/p + 1/p' = 1$. By using the inverse estimates \cite[Corollary 1.141]{EG2004} and \eqref{eq:ML_norm}, we obtain
 \[
  \module{\langle (v_h - \cK_h^{-1} v_h), \phi_h \rangle_h } \leq C_p h \norm{Dv_h}{L^p(\Omega)} \norm{\phi_h}{L^{p'}(\Omega)} \leq C_p' h \norm{Dv_h}{L^p(\Omega)} \norm{\phi_h}{h,p'}.
 \]
 We conclude the proof by taking the supremum over $\phi_h \in V_h$.
\end{proof}

The semi-discrete heat equation with mass lumping reads as follows: find $v_h \in H^1(0,T;V_h)$ such that $v_h(0,\cdot) = v_0^h \in V_h$ and
\begin{equation}\label{eq:ML_heat}
 \int_0^t \langle \partial_t v_h(s), \phi_h(s) \rangle_h + \int_\Omega D v_h(s,x) \cdot D \phi_h(s,x) \, dx ds = \int_0^t \langle f, \phi_h\rangle_{V_h', V_h}\, ds
\end{equation}
for every $t \in (0,T)$ and $\phi_h \in L^2(0,T;V_h)$, where $f \in L^2(0,T;V_h')$. By a straightforward adaptation of \cite[Proposition 66.2]{EG2021} we have that the problem \eqref{eq:ML_heat} is well-posed for every $v_0^h \in V_h$ and $f \in L^2(0,T;V_h')$. Let $\Delta_h \colon V_h \to V_h$ be the discrete Laplacian defined by
\[
  \langle \Delta_h v_h, \phi_h \rangle_{L^2(\Omega)} = - \int_\Omega D v_h(x) \cdot D\phi_h(x) \, dx \quad \textnormal{for all } \phi_h \in V_h.
\]
Then, \eqref{eq:ML_heat} may be written
\[
  \cK_h v_h' - \Delta_h v_h = f \quad \textnormal{in } L^2(0,T;V_h'),
\]
or equivalently
\[
  v_h' - \widehat \Delta_h v_h = \cK_h^{-1} f \quad \textnormal{in } L^2(0,T;V_h'),
\]
where
\[
  \langle \widehat \Delta_h v_h, \phi_h \rangle_{h} = - \int_\Omega D v_h(x) \cdot D\phi_h(x) \, dx \quad \textnormal{for all } \phi_h \in V_h,
\]
or, equivalently, $\widehat \Delta_h = \cK_h^{-1} \Delta_h$.
We recall the following maximal regularity estimate for $\widehat \Delta_h$.
\begin{thm}[Maximal regularity for $\widehat \Delta_h$, {\cite[Theorem I]{KS2018}}]\label{thm:ML_max_reg}
  Assume \ref{h_parabolic_regularity_q_strong} and that the triangulations $\cT_h$ are of Delaunay type. Let $f \in L^p(0,T; V_h^p)$ for $1 < p \leq q$, where $q$ is fixed in \ref{h_parabolic_regularity_q_strong}. Then the unique solution to
  \[
   \begin{cases}
    v_h'(t) - \widehat \Delta_h v_h(t) = f(t) \quad \textnormal{for } t \in (0,T), \\
    v_h(0) = 0,
   \end{cases}
  \]
satisfies
  \[
    \norm{v_h}{L^p(V_h^p)} + \norm{v_h'}{L^p(V_h^p)} + \norm{\widehat \Delta_h v_h}{L^p(V_h^p)} \leq C \norm{f}{L^p(V_h^p)}
  \]
 for some $C > 0$ independent of $f$, $h$ and $T$.
\end{thm}

From \cref{thm:ML_max_reg}, we can deduce a first stability and error estimate for \eqref{eq:ML_heat}.
\begin{prop}\label{prop:ML_heat_estimates_1}
 Assume \ref{h_parabolic_regularity_q_strong} and let $1 < p \leq q$, where $q$ is fixed in \ref{h_parabolic_regularity_q_strong}. Let $v_h \in H^1(0,T;V_h)$ be the unique solution to \eqref{eq_mfg_parabolic_smooth_semi-discrete} with $v_h(0) = v_0^h \in V_h$ for some $f \in L^2(0,T;V_h')$, let $\widehat v_h \in H^1(0,T;V_h)$ be the unique solution to \eqref{eq:ML_heat} with $\widehat v_h(0) = \widehat v_0^h \in V_h$. Then, there exists $C > 0$, independent of $h$, such that
 \begin{equation}\label{eq:ML_heat_estimates_1}
  \norm{\widehat v_h}{W^{0,1}_p} \leq  C \left(\norm{v_0^h}{W^{1,p}} + \norm{f}{L^p(W^{-1,p})} + h^{-1}\norm{\widehat v_0^h - v_0^h}{L^p(\Omega)} \right)
 \end{equation}
and
 \begin{multline*}
   \norm{\widehat v_h - v_h}{L^p(Q_T)} + h \norm{D \widehat v_h - D v_h}{L^p(Q_T)} \\ \leq C \left( \norm{\widehat v_0^h - v_0^h}{L^p(\Omega)} + h\norm{D v_0^h}{L^p(\Omega)} + h \norm{D v_h}{L^p(Q_T)} \right).
  \end{multline*}
\end{prop}

\begin{proof}
  The argument is reminiscent of \cite[Corollary 2.2]{L2019}. Note first that we may write the equation satisfied by $v_h$ as
  \[
    \cK_h^{-1} v_h' - \widehat \Delta_h v_h = \cK_h^{-1} f \quad \textnormal{in } L^2(0,T;V_h').
  \]
  Therefore, setting $w_h(t) = \widehat v_h(t) - \cK_h^{-1} v_h(t) - e^{-t}\left(\widehat v_0^h - \cK_h^{-1} v_0^h \right)$, we have
  \[
  \begin{cases}
    w_h' - \widehat \Delta_h w_h = \widehat \Delta_h(\cK_h^{-1}v_h - v_h) + e^{-t} \left(\widehat v_0^h - \cK_h^{-1} v_0^h + \widehat \Delta_h \left(\widehat v_0^h - \cK_h^{-1} v_0^h \right) \right), \\
    w_h(0) = 0.
  \end{cases}
  \]
  Applying $\widehat \Delta_h^{-1}$ on both sides of the equation and setting $\theta_h = \widehat \Delta_h^{-1} w_h$, we obtain
  \[
  \begin{cases}
    \theta_h' - \widehat \Delta_h \theta_h = (\cK_h^{-1}v_h - v_h) + e^{-t} \left(\widehat v_0^h - \cK_h^{-1} v_0^h + \widehat \Delta_h^{-1} \left(\widehat v_0^h - \cK_h^{-1} v_0^h \right) \right), \\
    \theta_h(0) = 0.
  \end{cases}
  \]
  Since $v_h \in L^p(Q_T)$, we may apply \cref{thm:ML_max_reg} along with the $L^p$ stability of $\widehat \Delta_h^{-1}$ to deduce that
  \[
    \norm{w_h}{L^p(V_h^p)} = \norm{\widehat \Delta_h \theta_h}{L^p(V_h^p)} \leq C \left(\norm{\widehat v_0^h - \cK_h^{-1} v_0^h}{L^p(V_h^p)} + \norm{v_h - \cK_h^{-1} v_h}{L^p(V_h^p)} \right)
  \]
  for some $C > 0$ independent of $h$. Using \cref{lem:ML_quadrature_error}, the norm equivalence \eqref{eq:ML_norm} and the inverse estimates \cite[Corollary 1.141]{EG2004}, we obtain
  \begin{multline*}
   \norm{\widehat v_h - v_h}{L^p(Q_T)} + h \norm{D \widehat v_h - D v_h}{L^p(Q_T)} \\ \leq C \left( \norm{\widehat v_0^h - v_0^h}{L^p(\Omega)} + h\norm{D v_0^h}{L^p(\Omega)} + h \norm{D v_h}{L^p(Q_T)} \right).
  \end{multline*}
  In particular, we have
  \begin{equation}
  \begin{split}
   \norm{\widehat v_h}{W^{0,1}_p} & \leq \norm{\widehat v_h - v_h}{W^{0,1}_p} + \norm{v_h}{W^{0,1}_p} \\
   & \leq C \left(\norm{v_0^h}{W^{1,p}} + \norm{f}{L^p(W^{-1,p})} + h^{-1}\norm{\widehat v_0^h - v_0^h}{L^p(\Omega)} \right),
   \end{split}
  \end{equation}
  which proves \eqref{eq:ML_heat_estimates_1}.
\end{proof}

Assuming more regularity on the data, we are also able to derive the following stronger error estimate.
\begin{prop}\label{prop:ML_heat_estimates_2}
 Assume \ref{h_parabolic_regularity_q_strong} and let $1 < p \leq q$, where $q$ is fixed in \ref{h_parabolic_regularity_q_strong}. Let $v_h \in H^1(0,T;V_h)$ be the unique solution to \eqref{eq_mfg_parabolic_smooth_semi-discrete} with $v_h(0) = v_0^h \in V_h$ for some $f \in L^p(Q_T)$, let $\widehat v_h \in H^1(0,T;V_h)$ be the unique solution to \eqref{eq:ML_heat} with $\widehat v_h(0) = \widehat v_0^h \in V_h$. Then, there exists $C > 0$, independent of $h$, such that
 \[
   \norm{\widehat v_h - v_h}{W^{0,1}_p} \leq C\left(\norm{\widehat v_0^h - v_0^h}{W^{1,p}} + h \left( \norm{f}{L^p(Q_T)} + \norm{\Delta_h v_0^h}{L^p(\Omega)} \right) \right).
 \]
\end{prop}
\begin{proof}
  Since $f \in L^p(Q_T)$, from \cite[Theorem 2.1]{L2019}, we have
  \begin{equation}\label{eq:ML_heat_estimates_3}
   \norm{v_h'}{L^p(Q_T)} + \norm{\Delta_h v_h}{L^p(Q_T)} \leq C \left(\norm{f}{L^p(Q_T)} + \norm{\Delta_h v_0^h}{L^p}\right)
  \end{equation}
  for some $C > 0$ independent of $h$. Set $\theta_h := v_h - \widehat v_h$, so that
  \[
    \begin{cases}
     \theta_h' - \widehat \Delta_h \theta_h = v_h' - \cK_h^{-1} v_h', \\
     \theta_h(0) = v_0^h - \widehat v_0^h.
    \end{cases}
  \]
  It follows from \cref{prop:ML_heat_estimates_1} that
  \begin{equation}\label{eq:ML_heat_estimates_7}
   \norm{\theta_h}{W^{0,1}_p} \leq C \left(\norm{\widehat v_0^h - v_0^h}{W^{1,p}} + \norm{v_h' - \cK_h^{-1} v_h'}{L^p(W^{-1,p})} \right).
  \end{equation}
  Let $\phi \in W^{0,1}_{p'}$, with $1/p + 1/p' = 1$ and decompose
  \begin{multline}\label{eq:ML_heat_estimates_4}
   \left(v_h' - \cK_h^{-1} v_h', \phi \right)_{L^2(Q_T)} \\  = \left(v_h' - \cK_h^{-1} v_h', R_h \phi \right)_{L^2(Q_T)} + \left(v_h' - \cK_h^{-1} v_h', \phi - R_h \phi \right)_{L^2(Q_T)}.
  \end{multline}
  We first estimate
  \begin{align*}
    \left(v_h' - \cK_h^{-1} v_h', R_h \phi \right)_{L^2(Q_T)} & = \left( v_h', (I - \cK_h^{-1}) R_h \phi \right)_{L^2(Q_T)}.
  \end{align*}
  Using \cref{lem:ML_quadrature_error}, \cref{prop_parabolic_mfg_ritz} and \eqref{eq:ML_heat_estimates_3}, we have
  \begin{equation}\label{eq:ML_heat_estimates_5}
  \begin{split}
    \module{\left( v_h', (I - \cK_h^{-1}) R_h \phi \right)_{L^2(Q_T)}} & \leq \norm{v_h'}{L^p(Q_T)}\norm{(I - \cK_h^{-1})R_h \phi}{L^{p'}(Q_T)} \\
    & \leq C h \left(\norm{f}{L^p(Q_T)} + \norm{\Delta_h v_0^h}{L^p(\Omega)}\right) \norm{D R_h \phi}{L^{p'}(Q_T)} \\
    & \leq C h \left(\norm{f}{L^p(Q_T)} + \norm{\Delta_h v_0^h}{L^p(\Omega)}\right) \norm{\phi}{W^{0,1}_{p'}}.
  \end{split}
  \end{equation}
  On the other hand, also using \eqref{eq:K_h_equiv} and \eqref{eq:ML_heat_estimates_3}, we have
  \begin{equation}\label{eq:ML_heat_estimates_6}
  \begin{split}
   \module{\left(v_h' - \cK_h^{-1} v_h', \phi - R_h \phi \right)_{L^2(Q_T)}} & \leq C\norm{v_h'}{L^p(Q_T)} \norm{\phi - R_h \phi}{L^{p'}(Q_T)} \\
   & \leq Ch\left(\norm{f}{L^p(Q_T)} + \norm{\Delta_h v_0^h}{L^p(Q_T)}\right) \norm{\phi}{W^{0,1}_{p'}},
   \end{split}
  \end{equation}
  where we have used the fact that $ \norm{\phi(t) - R_h \phi(t)}{L^{p'}(\Omega)} \leq C h \norm{\phi(t)}{W^{1,{p'}}}$ since $p' \geq q/(q-1)$ and \ref{h_parabolic_regularity_q_strong} holds.
  Combining \eqref{eq:ML_heat_estimates_4}, \eqref{eq:ML_heat_estimates_5} and \eqref{eq:ML_heat_estimates_6}, we conclude that
  \[
    \norm{v_h' - \cK_h^{-1} v_h'}{L^p(W^{-1,p})} \leq Ch \left(\norm{f}{L^p(Q_T)} + \norm{\Delta_h v_0^h}{L^p(Q_T)}\right)
  \]
  for some $C > 0$ independent of $h$. We conclude the proof by injecting this last inequality in \eqref{eq:ML_heat_estimates_7}.
\end{proof}

  \subsubsection{Nonnegative approximation of the Fokker-Planck equation}

  Let $L_h \colon (0,T) \to \cL(V_h, V_h')$ be defined by
  \begin{equation}\label{eq:L_h}
   \langle L_h(t) v_h, \phi_h \rangle_{V_h',V_h} = \int_{\Omega} D v_h(x) \cdot D \phi_h(x) + v_h(x) b(t,x) \cdot D\phi_h(x) \, dx
  \end{equation}
  for all $(v_h,\phi_h) \in V_h \times V_h$, where $b \in L^\infty(Q_T)$. Our goal is to prove that the parabolic equation associated to $L_h$ preserves the nonnegativity of the initial condition. To this end, we are going to combine the mass lumping of the previous section with the artificial viscosity method considered in \cite{BE2002,OS2024a}. More precisely, we assume that the collection of triangulations $(\cT_h)_{h > 0}$ is \emph{strictly acute}, i.e., there exists $\alpha \in (0,1)$ such that for every pair of Lagrange basis functions $(\xi_h^i, \xi_h^j)$ and every element $K \in \cT_h$ such that $\mathring K \subset \supp(\xi_h^i) \cap \supp(\xi_h^j)$ and $i \neq j$, we have
  \begin{equation}
   D \xi_{h\vert K}^i \cdot D \xi_{h\vert K}^j \leq - \alpha \module{D \xi_{h\vert K}^i} \module{D \xi_{h\vert K}^j}.
  \end{equation}
  Note that, in particular, any strictly acute mesh is also of Delaunay type.
  Let us introduce the notation
  \[
   \Xi_h(K) := \left \{ \xi_h^i : \, \mathring K \subset \supp(\xi_h^i) \right\}
  \]
  and set
  \[
   \sigma_h := \min_{K \in \cT_h} \min_{\xi_h^i \in \Xi_h(K)} \left \{ \diam(K) \module{D \xi_{h\vert K}^i} \right \}.
  \]
  Note that since the triangulations $(\cT_h)_{h > 0}$ are shape-regular, there exist $\underline{\sigma},\, \overline{\sigma} > 0$ such that
  \[
   \underline{\sigma} \leq \sigma_h \leq \overline{\sigma} \quad \textnormal{for all } h > 0.
  \]
  Then, for every $h > 0$, we define $\gamma_h \in L^\infty(\Omega)$ and the stabilized linear operator $L_h^s \colon (0,T) \to \cL(V_h;V_h')$ by
  \begin{gather}
  \langle L_h^s(t) v_h, \phi_h \rangle_{V_h',V_h} = \int_{\Omega}(1 + \gamma_h(x)) D v_h(x) \cdot D \phi_h(x) + v_h(x) b(t,x) \cdot D\phi_h(x) \, dx \label{eq:L_h_stab}, \\
   \gamma_h(x) = \max \left \{ \frac{2 \norm{b}{L^\infty} \diam(K)}{\alpha \sigma_h} - 1, 0 \right\} \quad \textnormal{for all $x \in K$ and $K \in \cT_h$}.
  \end{gather}

  We have the following result.
  \begin{prop}\label{prop:max_principle}
   Assume that the collection of triangulations $(\cT_h)_{h > 0}$ is strictly acute. Let $b \in L^\infty(Q_T)$ and define $L_h^s \colon (0,T) \to \cL(V_h,V_h')$ by \eqref{eq:L_h_stab}. Then for every $v_h \in H^1(0,T;V_h)$ such that $v_h(0,\cdot) \geq 0$ and
   \begin{equation}\label{eq:prop_DMP}
    \int_0^t  \langle \partial_t v_h(s), \phi_h(s) \rangle_h + \langle L_h^s(s) v_h(s), \phi_h(s) \rangle_{V_h',V_h} ds =  0,
   \end{equation}
   for all $t \in (0,T)$ and $\phi_h \in L^2(0,T;V_h)$, we have $v_h(t,x) \geq 0$ for all $(t,x) \in Q_T$.
  \end{prop}
  \begin{proof}
   Note that \eqref{eq:prop_DMP} is equivalent to
   \[
    M_h^L \mathbf{v}_h'(t) +  A_h(t) \mathbf{v}_h(t) = 0 \quad \textnormal{for a.e. } t \in (0,T),
   \]
   where $A_h(t)$ is the stiffness matrix associated to $L_h^s(t)$, i.e.,
   \[
    \left(A_h(t) \right)_{i,j} = \langle L_h^s(t) \xi_h^j, \xi_h^i \rangle_{V_h',V_h}.
   \]
   By definition of $\gamma_h$, we have $\left(A_h(t)\right)_{i,j} \leq 0$ if $i \neq j$ and $\left(A_h(t)\right)_{i,i} > 0$. The conclusion then follows from an adaption of the proof of \cite[Theorem 15.5]{T2006}.
  \end{proof}

  \subsubsection{Stabilized semi-discrete MFG system}
  In this section, in addition to \ref{h_parabolic_mfg}, \ref{h_parabolic_mfg_diff}, we are going to assume that $H$ is $C^2$ with respect to both variables and that
  \begin{equation}\label{eq:H_lip}
  \begin{split}
   \module{H_p(x,p)} & \leq C_H \quad \textnormal{for all } (x,p) \in \Omega \times \RR^d, \\
   \module{H_{xp}(x,p)} & \leq C_H \left(1 + \module{p} \right) \quad \textnormal{for all } (x,p) \in \Omega \times \RR^d, \\
   \norm{dF[m](\rho)}{L^q(Q_T)} & \leq C_F \norm{\rho}{L^q(Q_T)} \quad \textnormal{for all } m \in L^\infty(Q_T), \, \rho \in L^q(Q_T).
  \end{split}
  \end{equation}

  Then, assuming that the triangulations $\cT_h$ are strictly acute and using the notations of the previous section, we define $\gamma_h \in L^\infty(\Omega)$ element-wise by
  \[
   \gamma_h(x) = \max \left \{ \frac{2 C_H \diam(K)}{\alpha \sigma_h} - 1, 0 \right\} \quad \textnormal{for all $x \in K$ and $K \in \cT_h$}.
  \]

  By a solution to the stabilized semi-discrete mean field game system, we mean a pair $(\widehat u_h, \widehat m_h) \in H^1(0,T;V_h) \times H^1(0,T;V_h)$ such that
\begin{multline}\label{eq:mfg_semi-discrete_HJB_stab}
 \int_0^t - \langle \partial_t \widehat u_h(s), \phi_h(s) \rangle_h + \int_{\Omega} (1 + \gamma_h(x)) D \widehat u_h(s,x) \cdot D\phi_h(s,x) + H(x,D \widehat u_h(s,x)) \phi_h(s,x) \, dx ds \\ = \int_0^t \int_\Omega F[ \widehat m_h(s)](x) \phi_h(s,x) \, dx ds
\end{multline}
for all $t \in (0,T)$ and $\phi_h \in L^2(0,T;V_h)$,
\begin{multline}\label{eq:mfg_semi-discrete_FP_stab}
 \int_0^t \langle \partial_t \widehat m_h(s), \psi_h(s) \rangle_h + \int_\Omega (1+\gamma_h(x)) D \widehat m_h(s,x) \cdot D\psi_h(s,x)\, dx ds \\ + \int_0^t \int_\Omega \widehat m_h(s,x) H_p(x,D \widehat u_h(s,x)) \cdot D\psi_h(s,x) \, dx ds = 0
\end{multline}
for all $t \in (0,T)$ and $\psi_h \in L^2(0,T;V_h)$, and if $\widehat u_h(T,\cdot) = u_T^h \in V_h$ and $ \widehat m_h(0,\cdot) = m_0^h \in V_h$.

 Let $\widehat v_h \in H^1(0,T;V_h)$ be the unique solution to
  \begin{multline}\label{eq:heat_stab}
    \int_0^t \langle \partial_t \widehat v_h(s), \phi_h(s) \rangle_h + \int_\Omega (1 + \gamma_h(x)) D \widehat v_h(s,x) \cdot D\phi_h(s,x) \, dx ds \\ = \int_0^t \langle f(s), \phi_h(s) \rangle_{V_h',V_h} \, ds
  \end{multline}
  for all $t \in (0,T)$ and $\phi_h \in L^2(0,T;V_h)$ with $\widehat v_h(0) = v_0^h \in V_h$, for some $f \in L^2(0,T;V_h')$. Since $\norm{\gamma_h}{L^\infty(Q_T)} = O(h)$, if \ref{h_parabolic_regularity_q_strong} holds, a straightforward application of \cref{prop:ML_heat_estimates_1} yields
  \[
   \norm{\widehat v_h}{W^{0,1}_q} \leq C \left(\norm{v_0^h}{W^{1,q}} + \norm{f}{L^q(W^{-1,q})} + h \norm{\widehat v_h}{W^{0,1}_q} \right),
  \]
  so that
  \begin{equation}
   \norm{\widehat v_h}{W^{0,1}_q} \leq C\left(\norm{v_0^h}{W^{1,q}} + \norm{f}{L^q(W^{-1,q})} \right)
  \end{equation}
  for $h > 0$ small enough. Moreover, if $v_h \in H^1(0,T;V_h)$ is the unique solution to \eqref{eq_mfg_parabolic_smooth_semi-discrete} with $v_h(0) = v_0^h$, then
  \begin{equation}\label{eq:heat_stab_estimate_1}
   \norm{\widehat v_h - v_h}{L^q} \leq C h\left(\norm{v_0^h}{W^{1,q}} + \norm{f}{L^q(W^{-1,q})} \right)
  \end{equation}
  for $h > 0$ small enough. Moreover, if $f \in L^q(Q_T)$,
  then, using \cref{prop:ML_heat_estimates_2}, we also have
  \begin{equation}\label{eq:heat_stab_estimate_2}
   \norm{\widehat v_h - v_h}{W^{0,1}_q} \leq Ch\left(\norm{\Delta_h v_0^h}{L^q(\Omega)} + \norm{f}{L^q(Q_T)} \right)
  \end{equation}
  for all $h > 0$ small enough.

  We can now define the linear operators $\widehat S_I^h$ and $ \widehat S_{IS}^h$ by $\widehat S_I^h v_0^h = \widehat v_h$, where $\widehat v_h$ solves \eqref{eq:heat_stab} with $f=0$, and $\widehat S^h_{IS} f =  \widehat u_h$, where $\widehat u_h$ solves \eqref{eq:heat_stab} with $v_0^h = 0$. We also consider $\widehat S^h_T$ and $ \widehat S_{TS}^h$ their backward analogues, i.e.,
\[
 \left(\widehat S^h_T v_0^h \right)(t,x) = \left(\widehat S^h_I v_0^h \right)(T-t,x) \quad  \textnormal{and} \quad \left(\widehat S^h_{TS} f \right)(t,x) = \left(\widehat S^h_{IS} f \right)(T - t, x).
\]
The preceding discussion then yields the following result.
\begin{lem}\label{lem:heat_stab_estimates}
 Assume that \ref{h_parabolic_regularity_q_strong} holds and let $q > 2(d+2)$ be as in \ref{h_parabolic_regularity_q_strong}. Then there exist $h_1 > 0$ and $C > 0$ such that
 \begin{gather*}
  \norm{\left(\widehat S_I^h - S_I^h\right) v_0^h}{L^q(Q_T)} \leq C h \norm{v_0^h}{W^{1,q}} \\ \norm{\left(\widehat S_T^h - S_T^h\right) v_0^h}{W^{0,1}_q} \leq C h \norm{\Delta_h v_0^h}{L^q(\Omega)}, \\
  \norm{\widehat S_{IS}^h - S_{IS}^h}{\cL(L^{q}(W^{-1,q}), L^q)} + \norm{\widehat S_{TS}^h - S_{TS}^h}{\cL(L^{q}, W^{0,1}_q)} \leq C h
 \end{gather*}
 for all $0 < h \leq h_1$.
\end{lem}

In order to prove the existence of solutions to \eqref{eq:mfg_semi-discrete_HJB_stab}-\eqref{eq:mfg_semi-discrete_FP_stab}, let $\widehat \Upsilon_h \colon X \to X$, with $X$ is defined in \eqref{eq_mfg_parabolic_smooth_spaces}, be defined by
\begin{equation}
 \widehat \Upsilon_h(v,\rho) = \left(I - \widehat \cS_1^h \cR_1^h - \widehat \cS_2^h \cR_2\right)(v,\rho)
\end{equation}
where
\begin{equation}
 \widehat \cS_1^h := \begin{pmatrix} \widehat S_T^h & 0 \\ 0 & \widehat S_I^h \end{pmatrix}, \quad \widehat \cS_2^h := \begin{pmatrix} \widehat S_{TS}^h & 0 \\ 0 & \widehat S_{IS}^h \end{pmatrix}
\end{equation}
and $\cR_1^h$ and $\cR_2$ are defined in \eqref{eq:R_1_h} and \eqref{eq_mfg_parabolic_smooth_def_R}, respectively. Then $(\widehat u_h, \widehat m_h)$ is a solution to \eqref{eq:mfg_semi-discrete_HJB_stab}-\eqref{eq:mfg_semi-discrete_FP_stab} if and only if $\widehat \Upsilon_h(\widehat u_h, \widehat m_h) = 0$.

We are now able to prove the following error estimate for the stabilized MFG system which is reminiscent of \cite{OS2025b}.
\begin{thm}\label{thm:error_estimate_parabolic_stab}
   Assume \ref{h_parabolic_mfg}, \ref{h_parabolic_regularity_q_strong}, \ref{h_parabolic_mfg_diff} and \eqref{eq:H_lip} and let $q > 2(d+2)$ as in \ref{h_parabolic_regularity_q_strong}. Assume also that the triangulations $(\cT_h)_{h > 0}$ are strictly acute. Let $(u,m) \in X$ be a stable solution to \eqref{eq_mfg_parabolic_smooth} such that $m \in L^\infty(Q_T) \cap W^{0,1}_q(Q_T)$. Assume also that
    \begin{equation}\label{eq:thm_error_estimate_parabolic_stab_1}
     \lim_{h \to 0} \norm{m_0 - m_0^h}{L^q(\Omega)} + \norm{u_T - u_T^h}{W^{1,q}} = 0
    \end{equation}
    and
    \begin{equation}\label{eq:thm_error_estimate_parabolic_stab_2}
     \norm{m_0^h}{W^{1,q}} +  \norm{\Delta_h u_T^h}{L^q(\Omega)} \leq K \quad \textnormal{for all } h > 0,
    \end{equation}
    for some $K > 0$.
    Then there exist $h_0 > 0$ and a neighborhood $\mathcal{O}$ of $(u,m)$ in $X$, such that, for every $0 < h \leq h_0$, there exists $(\widehat u_h, \widehat m_h) \in \mathcal{O}$ such that $\widehat \Upsilon_h(\widehat u_h, \widehat m_h) = 0$. Moreover, $(\widehat u_h, \widehat m_h)$ is the unique zero of $\widehat \Upsilon_h$ in $\mathcal{O}$ and we have the error estimate
   \begin{multline}\label{eq:error_estimate_parabolic_stab}
   \norm{u - \widehat u_h}{W^{0,1}_q} + \norm{m - \widehat m_h}{L^q(Q_T)} \\ \leq C \left( \norm{P_h m_0 - m_0^h}{L^q(\Omega)} + h^{-1}\norm{P_h u_T - u_T^h}{L^q(\Omega)} \right) \\
    + C \left[ \inf_{(v_h,\rho_h) \in L^q(0,T;V_h)^2} \norm{u - v_h}{W^{0,1}_q} + h \norm{m - \rho_h}{W^{0,1}_q} \right] + O(h).
   \end{multline}
   Finally, if $m_0^h \geq 0$, then $\widehat m_h(t,\cdot) \geq 0$ for all $t \in [0,T]$.
\end{thm}

\begin{rem}\label{rem:error_estimate_parabolic_stab}
 Note that assumption
 \[
  \sup_{h > 0} \norm{\Delta_h u_T^h}{L^q(\Omega)} < + \infty
 \]
 holds for instance if $u_T \in H^1_0(\Omega) \cap W^{2,q}(\Omega)$ and $u_T^h = R_h u_T$ or $u_T^h = \cI_h u_T$. Indeed, in the first case we have
 \begin{align*}
  \langle \Delta_h u_T^h, \phi_h \rangle_{L^2(\Omega)} & = - \int_{\Omega} D u_T^h(x) \cdot D \phi_h(x) \, dx = - \int_\Omega D u_T(x) \cdot D \phi_h(x) \, dx  \\
  & = \langle \Delta u_T, \phi_h \rangle_{L^2(\Omega)}
 \end{align*}
 for every $\phi_h \in V_h$, so that $\Delta_h u_T^h = P_h (\Delta u_T)$. The inequality then follows from the $L^p$ stability of $P_h$ in \cref{prop_parabolic_mfg_proj}. For the second case, observe that
 \[
  \Delta_h u_T^h = P_h \left(\Delta u_T \right) + \Delta_h \left(\cI_h u_T - \cR_h u_T \right).
 \]
 Set $w_h := \cI_h u_T - \cR_h u_T$ and $z_h := \module{\Delta_h w_h}^{q-1} \sgn\left(\Delta_h w_h \right)$. Then
 \[
    \langle \Delta_h w_h, P_h z_h \rangle_{L^2(\Omega)} = \langle \Delta_h w_h, z_h \rangle_{L^2(\Omega)} = \norm{\Delta_h w_h}{L^q(\Omega)}^q.
 \]
 On the other hand, using inverse estimates and the $L^p$-stability of $P_h$, we have
 \[
  \langle \Delta_h w_h, P_h z_h \rangle_{L^2(\Omega)} \leq C h^{-2} \norm{w_h}{L^q(\Omega)} \norm{z_h}{L^{q'}(\Omega)} \leq \frac{1}{2} \norm{\Delta_h w_h}{L^q(\Omega)}^q + Ch^{-2q} \norm{w_h}{L^q(\Omega)}^q,
 \]
 so that
 \[
  \norm{\Delta_h w_h}{L^q(\Omega)} \leq C h^{-2} \norm{w_h}{L^q(\Omega)}.
 \]
 The conclusion then follows from the fact that $\norm{w_h}{L^q(\Omega)} \leq C h^2 \norm{u_T}{W^{2,q}}$.
\end{rem}

\begin{proof}[Proof of \cref{thm:error_estimate_parabolic_stab}]
 The nonnegativity of $\widehat m_h$ is a direct consequence of \cref{prop:max_principle}. We turn to the proof of the error estimate.
 Notice first that under assumption \eqref{eq:H_lip}, we have $\cR_2(X)\subset L^q(Q_T) \times L^q(0,T; W^{-1,q})$.
 Observe that \cref{lem_mfg_parabolic_errors,lem:heat_stab_estimates} imply that
 \begin{align*}
  \widehat \Upsilon_h(u,m) - \Upsilon(u,m) & = \cS_1 \cR_1(u,m) - \cS_1^h \cR_1^h(u,m) +  \left(\cS_1^h - \widehat \cS_1^h \right) \cR_1^h(u,m) \\ & \quad + \left(\cS_2 - \cS_2^h \right) \cR_2(u,m) + \left(\cS_2^h - \widehat \cS_2^h \right) \cR_2(u,m) \xrightarrow{h \to 0} 0.
 \end{align*}
  Moreover, we deduce from \cref{prop_parabolic_mfg_diff} that $\widehat \Upsilon_h$ is strictly differentiable at $(u,m)$ with
  \[
   d\widehat \Upsilon_h[u,m] = I - \widehat \cS_2^h \circ d\cR_2[u,m] = d\Upsilon[u,m] + \left(\cS_2 - \cS_2^h \right) \circ d \cR_2[u,m] + \left(\cS_2^h - \widehat \cS_2^h \right) \circ d \cR_2[u,m].
  \]
  Since $d\cR_2[u,m](X) \subset L^q(Q_T) \times L^q(0,T; W^{-1,q})$ as a consequence of \eqref{eq:H_lip}, \cref{lem:heat_stab_estimates} then also yields
  \[
   \norm{d \widehat \Upsilon_h[u,m] - d \Upsilon[u,m]}{\cL(X)} = O(h).
  \]
  We may therefore apply \cref{thm:reworked_BRR} to deduce that there exists a neighborhood $\mathcal{O}$ of $(u,m)$ in $X$ and a unique $(\widehat u_h, \widehat m_h) \in \mathcal{O}$ such that $\widehat \Upsilon_h(\widehat u_h, \widehat m_h) = 0$ and
  \begin{align*}
   \norm{\widehat u_h - u}{W^{0,1}_q} + \norm{\widehat m_h - m}{L^q} & \leq C \norm{ \widehat \Upsilon_h(u,m) - \Upsilon(u,m)}{X} \\
   & \leq C \left(\norm{ \widehat \Upsilon_h(u,m) - \Upsilon_h(u,m)}{X} + \norm{ \Upsilon_h(u,m) - \Upsilon(u,m)}{X} \right).
  \end{align*}
  The estimate \eqref{eq:error_estimate_parabolic_stab} then follows from \cref{cor:quasi_optimal_q} and \cref{lem:heat_stab_estimates}.
\end{proof}

\begin{cor}\label{cor:error_estimate_parabolic_stab}
   Under the assumptions and notations of \cref{thm:error_estimate_parabolic_stab} and if $(u,m) \in W^{1,2}_q(Q_T) \times \cH^1_q(Q_T)$ and
    \begin{equation}\label{eq:cor_error_estimate_parabolic_stab_1}
     \norm{\Delta_h m_0^h}{L^q(\Omega)} \leq K \quad \textnormal{for all } h > 0,
    \end{equation}
    for some $K > 0$, then we have the error estimate
   \begin{multline}\label{eq:error_estimate_parabolic_true_quasi_optimal_stab}
   \norm{u - \widehat u_h}{W^{0,1}_q} + \norm{m - \widehat m_h}{W^{0,1}_q} \\ \leq Ch^{-1} \left(\norm{P_h m_0 - m_0^h}{L^q(\Omega)} + \norm{P_h u_T - u_T^h}{L^q(\Omega)} \right) \\
    + C \left[ \inf_{(v_h,\rho_h) \in L^q(0,T;V_h)^2} \norm{u - v_h}{W^{0,1}_q} + \norm{m - \rho_h}{W^{0,1}_q} \right] + O(h).
   \end{multline}
\end{cor}

\begin{rem}
 From \cref{rem:error_estimate_parabolic_stab}, we know that \eqref{eq:cor_error_estimate_parabolic_stab_1} holds for $m_0^h = \cI_h m_0$ if $m_0 \in H^1_0(\Omega) \cap W^{2,q}(\Omega)$. Moreover, since we consider $\PP^1$-Lagrange finite elements, this choice also ensures that $m_0^h \geq 0$ if $m_0 \geq 0$.
\end{rem}

\begin{proof}[Proof of \cref{cor:error_estimate_parabolic_stab}]
 Let $\widetilde m_h \in H^1(0,T;V_h)$ be defined by \eqref{eq:tilde_m_h} and $\widehat \rho_h \in H^1(0,T;V_h)$ be the unique solution to
 \begin{multline}
 \int_0^t \langle \partial_t \widehat \rho_h(s), \phi_h(s) \rangle_h + \int_{\Omega}  D \widehat \rho_h(s,x) \cdot D \phi_h(s,x) \, dx ds \\ = - \int_0^t \int_\Omega m(s,x) H_p(x,Du) \cdot D \phi_h(s,x) \, dx ds
\end{multline}
for all $t \in (0,T)$ and $\phi_h \in L^2(0,T;V_h)$ with $\widehat \rho_h(0) = m_0^h$. We then decompose
\begin{equation}\label{eq:quasi_optimal_stab_1}
 \norm{m - \widehat m_h}{W^{0,1}_q} \leq \norm{m - \widetilde m_h}{W^{0,1}_q} + \norm{\widetilde m_h - \widehat \rho_h}{W^{0,1}_q} + \norm{\widehat \rho_h - \widehat m_h}{W^{0,1}_q}.
\end{equation}
From \cref{cor:quasi_optimal_q}, we have
\begin{equation}\label{eq:quasi_optimal_stab_2}
 \norm{m - \widetilde m_h}{W^{0,1}_q} \leq C \left(h^{-1} \norm{P_h m_0 - m_0^h}{L^q(\Omega)} + \inf_{\rho_h \in L^q(0,T;V_h)} \norm{m - \rho_h}{W^{0,1}_q} \right).
\end{equation}
Moreover, since $(u,m) \in W^{1,2}_q(Q_T) \times \cH^1_q(Q_T)$ and using \eqref{eq:H_lip}, we obtain that
\[
 \diver \left(m H_p(x,Du) \right) = Dm \cdot H_p(x,Du) + m \tr\left(H_{xp}(x,Du) + H_{pp}(x,Du) D^2u \right)
\]
belongs to $L^q(Q_T)$. It then follows from \cref{prop:ML_heat_estimates_2} that
\begin{equation}\label{eq:quasi_optimal_stab_3}
 \norm{\widetilde m_h - \widehat \rho_h}{W^{0,1}_q}  = O(h).
\end{equation}
Finally, arguing as in the proof of \cref{cor:quasi_optimal} and using \cref{prop:ML_heat_estimates_1}, we have
\begin{equation}\label{eq:quasi_optimal_stab_4}
 \norm{\widehat \rho_h - \widehat m_h}{W^{0,1}_q} \leq C \left(\norm{m - \widehat m_h}{L^q(Q_T)} + \norm{Du - D\widehat u_h}{L^q(Q_T)} \right) + O(h).
\end{equation}
The error estimate \eqref{eq:error_estimate_parabolic_true_quasi_optimal_stab} then follows from \eqref{eq:error_estimate_parabolic_stab}, \eqref{eq:quasi_optimal_stab_1}, \eqref{eq:quasi_optimal_stab_2}, \eqref{eq:quasi_optimal_stab_3} and \eqref{eq:quasi_optimal_stab_4}.
\end{proof}

\appendix

\section{Nemytskii operators}
\label{section:nemytskii}

Let $(\Omega,\cF, \mu)$ be a finite complete measure space and assume that $f \colon \Omega \times \RR^n \to \RR$ is a Carathéodory function, \textit{i.e.}, measurable with respect to the first variable and continuous with respect to the second one. We may define a mapping $\mathfrak{f}$ from $\mathscr{F}(\Omega, \RR^n)$, the vector space of function defined on $\Omega$ with valued in $\RR^d$, to $\mathscr{F}(\Omega, \RR)$ by setting
\[
 \mathfrak{f}[u] = f(\cdot,u(\cdot)) \quad \textnormal{for all } u \in \mathscr{F}(\Omega,\RR^n).
\]
This mapping is often called the \emph{Nemytskii operator}\index{Nemytskii operator}\footnote{Sometimes also called \emph{superposition operator}\index{superposition operator}.} associated to $f$.
If
\begin{equation}\label{eq:Nemytskii_growth}
 \module{f(x,z)} \leq C \left (1 + \module{z}^{\frac{p}{q}} \right) \quad \textnormal{for all } (x,z) \in \Omega \times \RR^n,
\end{equation}
for $p,\, q \in [1,\infty)$ and some $C > 0$, then it is easy to see that the Nemytskii operator $\mathfrak{f}$ maps $L^p(\Omega,\cF,\mu)$ to $L^q(\Omega,\cF,\mu)$.
We have the following result.
\begin{thm}[{\cite[Theorems 3.7, 3.10 and 3.13]{AZ1990}}]\label{thm:nemytskii_lebesgue}
 Let $p,q \in [1,\infty)$ and assume that \eqref{eq:Nemytskii_growth} holds. Then the Nemytskii operator $\mathfrak{f} \colon L^{p}(\Omega,\cF,\mu) \to L^q(\Omega, \cF, \mu)$ is continuous. Moreover,
 \begin{enumerate}[label={\rm(\roman*)}]
  \item $\mathfrak{f}$ is Lipschitz continuous if $q \leq p$ and
  \[
   \module{f(x,z_1) - f(x,z_2)} \leq L \module{z_1 - z_2} \quad \textnormal{for all } z_1, \, z_2 \in \RR^n,
  \]
  with $\Lip(\mathfrak{f}) \leq L \mu(\Omega)^{\frac{p-q}{pq}}$;
  \item $\mathfrak{f}$ is continuously differentiable if $q < p$ and $f$ is $C^1$ with respect to its second variable with
  \[
   \module{D_z f(x,z)} \leq C \left(1 + \module{z}^{\frac{p}{q} -1} \right) \quad \textnormal{for all } (x,z) \in \Omega \times \RR^n,
  \]
  and we have
  \[
   d \mathfrak{f}[u](v) = \mathfrak{f_z}[u] \cdot v,
  \]
  where $\mathfrak{f_z} \colon L^p(\Omega,\cF,\mu) \to L^{\frac{pq}{p-q}}(\Omega, \cF, \mu; \RR^n)$ is the Nemytskii operator associated to $D_z f$.
 \end{enumerate}
\end{thm}

For a locally Lipschitzian function $\phi \colon \RR^n \to \RR$, the \emph{Clarke generalized directional derivative} $\phi^\circ \colon \RR^n \times \RR^n \to \RR$ is defined by
  \begin{equation}
   \phi^\circ(z,\upsilon) = \limsup_{\substack{y \to z \\ \tau \searrow 0}} \frac{\phi(y + \tau \upsilon) - \phi(y)}{\tau}.
  \end{equation}
The \emph{Clarke's subdifferential} of $\phi$ at $z \in \RR^n$ is then defined as the set
\begin{equation}\label{eq:subdiff_def}
 \partial^C \phi(z) = \left \{ \xi \in \RR^n : \phi^\circ(z,\upsilon) \geq \xi \cdot \nu \, \textnormal{ for all } \nu \in \RR^n \right \},
\end{equation}
but is also characterized by
\begin{equation}\label{eq:clarke_characterization}
 \partial^C \phi(z) = \co \left \{ \xi \in \RR^{n} : \exists (z_k)_{k \geq 0} \subset \RR^n \setminus \mathcal{N}, \, \xi = \lim_{k \to \infty} D \phi(z_k) \right \},
\end{equation}
where $\cN \subset \RR^n$ is a negligible set\footnote{The set $\mathcal{N}$ is known to exist because of Rademacher's theorem \cite[Theorem 3.2]{EG2015}. Although it may seem like this definition depends on the choice of the set $\mathcal{N}$, it was proved in \cite[Theorem 4]{W1981} that this is not the case.}. It is known that Clarke's subdifferential is a nonempty, compact and convex set for all $z \in \RR^n$. We refer to \cite{C1990} for further information on this topic.
By extension, if the function $f \colon \Omega \times \RR^n \to \RR$ is locally Lipschitz continuous with respect to its second variable, we denote by $\partial^C f(x,z)$ the Clarke subdifferential of $f(x,\cdot)$ at $z \in \RR^n$.

\begin{prop}[{\cite[Theorem 4.11]{BLS2025a}}]\label{thm:nemytskii_mean_value}
Let $1 \leq q < p < \infty$. Assume that $f$ is locally Lipschitzian with respect to its second variable and satisfies \eqref{eq:Nemytskii_growth} and
\[
 \xi \in \partial^C f(x,z) \Longrightarrow \module{\xi} \leq C \left(1 + \module{z}^{\frac{p}{q} - 1} \right).
\]
Then, for every $u, \, v \in L^p(\Omega,\cF, \mu)$, there exists $\xi \in L^r(\Omega, \cF, \mu)$ such that
\[
 \mathfrak{f}[u] - \mathfrak{f}[v] = \xi \cdot (u - v) \quad \textnormal{$\mu$-a.e.,}
\]
where $r = \frac{pq}{p-q}$.
\end{prop}

\section{Proofs}
\label{section:proofs}

\subsection{Proof of \cref{prop_comparison_principle}}
\label{section:proof_comparison}
Let $w = u-v$. We have to prove that $w \leq 0$ almost everywhere. Using the fact that $u$ and $v$ are sub- and supersolutions to \eqref{eq_mfg_parabolic_smooth_HJ}, respectively, we have
  \begin{multline}
   \int_0^t \langle \partial_t w(s), \phi(s) \rangle_{H^{-1},H^1} + \int_\Omega D w(s,x) \cdot D \phi(s,x) \, dx ds \\ \leq \int_0^t \int_\Omega \left( H(x,Dv(s,x)) - H(x,Du(s,x)) \right) \phi(s,x) \, dx ds
  \end{multline}
  for all $\phi \in C^\infty_c([0,T] \times \Omega)$ with $\phi \geq 0$ and $t \in (0,T)$. From \cref{thm:nemytskii_mean_value}, there exists $\xi \in L^{d+2}(Q_T;\RR^d)$ such that
   \[
    H(x,Du(t,x)) - H(x,Dv(t,x)) = - \xi(t,x) \cdot Dw(t,x) \quad \textnormal{for a.e. $(t,x) \in Q_T$},
   \]
  so that
  \begin{multline}\label{eq_mfg_parabolic_smooth_HJ_comparison_1}
   \int_0^t \langle \partial_t w(s), \phi(s) \rangle_{H^{-1},H^1} + \int_\Omega D w(s,x) \cdot D \phi(s,x) \, dx ds \\ \leq \int_0^t \int_\Omega \xi(s,x) \cdot Dw(s,x) \phi(s,x) \, dx ds
  \end{multline}
  for all $\phi \in C^\infty_c([0,T] \times \Omega)$ with $\phi \geq 0$ and $t \in (0,T)$. Observe that, by density, \eqref{eq_mfg_parabolic_smooth_HJ_comparison_1} also holds for any $\phi \in W^{0,1}_{2}(Q_T) \cap L^{2(d+2)/d}(Q_T)$. Using \cite[Theorem 6.9]{L1996}, we know that $w \in L^{2(d+2)/d}(Q_T)$ and hence $(w)_+ \in  W^{0,1}_{2}(Q_T) \cap L^{2(d+2)/d}(Q_T)$ with $D(w(t,x))_+ = Dw(t,x) \mathbbm{1}_{\{w(t,x) > 0 \}}$. It follows that
  \begin{multline}\label{eq_mfg_parabolic_smooth_HJ_comparison_2}
   \int_0^t \langle \partial_t w(s), (w(s))_+ \rangle_{H^{-1},H^1} + \int_\Omega \module{D w(s,x)}^2 \mathbbm{1}_{\{w > 0 \}}(s,x) \, dx ds \\ \leq \int_0^t \int_\Omega \xi(s,x) \cdot Dw(s,x) (w(s,x))_+ \, dx ds
  \end{multline}
  for all $t \in (0,T)$. Using Hölder's and Young's inequalities, we then deduce that
  \begin{multline}\label{eq_mfg_parabolic_smooth_HJ_comparison_3}
   \int_0^t \langle \partial_t w(s), (w(s))_+ \rangle_{H^{-1},H^1} + \frac{1}{2} \int_\Omega \module{D w(s,x)}^2 \mathbbm{1}_{\{w > 0 \}}(s,x) \, dx ds \\ \leq C \left(\int_0^t \int_\Omega \module{\xi(s,x)}^{d+2}\, dx ds \right)^{\frac{2}{d+2}} \left(\int_0^t \int_\Omega \module{(w(s,x))_+}^{\frac{2(d+2)}{d}} \, dx ds \right)^{\frac{2d}{2(d+2)}}
  \end{multline}
  for all $t \in (0,T)$. From the parabolic Sobolev inequality \cite[Theorem 6.9]{L1996}, we have
  \begin{multline}
   \left(\int_0^t \int_\Omega \module{(w(s,x))_+}^{\frac{2(d+2)}{d}} \, dx ds \right)^{\frac{d}{2(d+2)}} \\ \leq C \left(\sup_{s \in [0,t]} \norm{(w(s,\cdot))_+}{L^2(\Omega)} \right)^{\frac{2}{d+2}} \left( \int_0^t \int_\Omega \module{Dw(s,x)}^2 \mathbbm{1}_{\{w > 0\}}(s,x)\, dx ds \right)^{\frac{d}{2(d+2)}}.
  \end{multline}
  Another application of Young's inequality then yields
  \begin{multline}
   \int_0^t \langle \partial_t w(s), (w(s))_+ \rangle_{H^{-1},H^1} + \frac{1}{4} \int_\Omega \module{D w(s,x)}^2 \mathbbm{1}_{\{w > 0 \}}(s,x) \, dx ds \\ \leq C \left(\int_0^t \int_\Omega \module{\xi(s,x)}^{d+2}\, dx ds \right) \left(\sup_{s \in [0,t]} \norm{(w(s,\cdot))_+}{L^2(\Omega)} \right)^2.
  \end{multline}
  In particular, since $(w(0,\cdot))_+ = 0$, we have
  \[
   \sup_{s \in[0,t]} \norm{(w(s,\cdot))_+}{L^2(\Omega)}^2 \leq C \left(\int_0^t \int_\Omega \module{\xi(s,x)}^{d+2}\, dx ds \right)\left(\sup_{s \in [0,t]} \norm{(w(s,\cdot))_+}{L^2(\Omega)} \right)^2
  \]
  for all $t \in (0,T)$. Let $\tau > 0$, depending only on $\xi$ and $C$, be such that
  \[
    \left(\int_t^{t + \tau} \int_\Omega \module{\xi(s,x)}^{d+2}\, dx ds \right)^{\frac{1}{2}} \leq \frac{1}{2C}
  \]
  for all $t \in (0 ,T-\tau)$. We then have
  \[
   \sup_{s \in[0,\tau]} \norm{(w(s,\cdot))_+}{L^2(\Omega)} = 0.
  \]
  We may then proceed by induction to obtain that
  \[
   \sup_{s \in[0,T]} \norm{(w(s,\cdot))_+}{L^2(\Omega)} = 0,
  \]
  which concludes the proof.

\subsection{Proof of \cref{thm_mfg_parabolic_HJ_wp}}
\label{section:proof_HJ_wp}

We start with the following estimate, which is a consequence of the comparison principle for \eqref{eq_mfg_parabolic_smooth_HJ}.
\begin{lem}\label{cor_comparison_principle}
  Let $u$ be a weak solution to \eqref{eq_mfg_parabolic_smooth_HJ} with $Du \in L^{d+2}(Q_T)$, then
  \[
   \norm{u}{L^\infty(Q_T)} \leq \norm{u_0}{L^\infty(\Omega)} + \left( C_H + \norm{f}{L^\infty(Q_T)} \right)T.
  \]
  Moreover, there exist $\alpha \in (0,1)$ and $K > 0$ such that $u \in C^{\alpha/2, \alpha}(Q_T)$ with
  \[
   \norm{u}{\alpha/2, \alpha} \leq K.
  \]
\end{lem}
\begin{proof}
 The conclusion follows from \cref{prop_comparison_principle} by noticing that
 \[
  v(t,x) = \norm{u_0}{L^\infty(\Omega)} + \left(C_H + \norm{f}{L^\infty(Q_T)} \right) t
 \]
 and $w(t,x) = -v(t,x)$ are super- and subsolutions to \eqref{eq_mfg_parabolic_smooth_HJ}, respectively. The Hölder estimate then follows from \cite[Theorem 1.1 p.419]{LSU1968}.
\end{proof}

We now turn to the proof of \cref{thm_mfg_parabolic_HJ_wp}. We define a mapping $\Phi \colon W^{0,1}_{q}(Q_T) \to W^{0,1}_{q}(Q_T)$ by setting $\Phi(w) = v$, where $v$ is the unique weak solution to
 \[
  \begin{cases}
   \partial_t v - \Delta v = f - H(x,Dw) \quad & \textnormal{in } Q_T, \\
   v(t,\cdot) = 0 \quad & \textnormal{on $\partial\Omega$ for $t \in (0,T)$,} \\
   v(0,\cdot) = u_0 \quad & \textnormal{on } \Omega.
  \end{cases}
 \]
 Let us check that $v \in W^{0,1}_q(Q_T)$. From \eqref{eq_mfg_parabolic_smooth_HJ_growth}, we have that $H(\cdot, Dw) \in L^{q/2}(Q_T)$. It then follows from \cref{prop_parabolic_FP_DGNM,prop_parabolic_W2p} that $v \in W^{1,2}_{q/2}(Q_T) \cap L^\infty(Q_T)$. From the Gagliardo-Nirenberg inequality \cite{N1966} we have
 \[
  \norm{Dv(t,\cdot)}{L^q(\Omega)} \leq C_1\norm{D^2 v(t,\cdot)}{L^{q/2}(\Omega)}^{1/2} \norm{v(t,\cdot)}{L^\infty(\Omega)}^{1/2} + C_2 \norm{v(t,\cdot)}{L^\infty(\Omega)}
 \]
 and hence
 \begin{equation}\label{eq_mfg_parabolic_smooth_GN}
  \int_0^T \norm{Dv(t,\cdot)}{L^q(\Omega)}^q \, dt \leq C_1 \norm{v}{L^\infty(Q_T)}^{q/2} \int_0^T \norm{D^2 v(t,\cdot)}{L^{q/2}(\Omega)}^{q/2} \, dt + C_2T \norm{v}{L^\infty(Q_T)}^q,
 \end{equation}
 so that $Dv \in L^{q}(Q_T)$. From a version of the Aubin-Dubinskii lemma \cite[Corollary 8]{S1987}, we have that the range of $\Phi$ is relatively compact in $C([0,T];W^{1,q/2}(\Omega)) \hookrightarrow L^\infty(Q_T)$ and it then follows from \eqref{eq_mfg_parabolic_smooth_GN} that $\Phi$ is compact. In order to apply the Leray-Schauder theorem \cite[Theorem 11.3]{GT2001}, we have to prove that the set of those $u$ such that
 $u = \sigma \Phi(u)$ for some $\sigma \in [0,1]$ is bounded in $W^{0,1}_{q}(Q_T)$. If $u = \sigma \Phi(u)$, then $u$ is a weak solution to
 \[
  \begin{cases}
   \partial_t u - \Delta u = \sigma \left(f - H(x,Du) \right) \quad & \textnormal{in } Q_T, \\
   u(t,\cdot) = 0 \quad & \textnormal{on $\partial\Omega$ for $t \in (0,T)$,} \\
   u(0,\cdot) = \sigma u_0 \quad & \textnormal{on } \Omega.
  \end{cases}
 \]
 From \cref{cor_comparison_principle}, we have that
 \[
  \norm{u}{L^\infty(Q_T)} \leq \norm{u_0}{L^\infty(\Omega)} + \left( C_H + \norm{f}{L^\infty(Q_T)} \right)T.
 \]
 Therefore, using \eqref{eq_mfg_parabolic_smooth_GN}, it appears that it is enough to prove that $u$ is bounded in $W^{1,2}_{q/2}(Q_T)$. We follow an argument from \cite{CG2021}. From \cref{prop_parabolic_W2p}, we have that
 \begin{equation}\label{eq_mfg_parabolic_smooth_HJ_Lp_estimate}
  \norm{u}{W^{1,2}_{q/2}} \leq C \left(1 + \norm{Du}{L^{q}(Q_T)}^2 + \norm{f}{L^{q/2}(Q_T)} + \norm{u_0}{W^{2,q/2}} \right)
 \end{equation}
 for some $C >0$ independent of $\sigma \in [0,1]$. Moreover, from \cref{cor_comparison_principle}, we also have
 \[
  \norm{u}{\alpha/2,\alpha} \leq K
 \]
for some $K > 0$ and $\alpha \in (0,1)$ independent of $\sigma \in [0,1]$. The Miranda-Nirenberg inequality \cite{N1966} then yields
 \[
  \norm{Du(t,\cdot)}{L^q(\Omega)} \leq C_1 \norm{D^2 u(t,\cdot)}{L^{q/2}(\Omega)}^{\theta} \norm{u(t,\cdot)}{\alpha/2,\alpha}^{1-\theta} + C_2 \norm{u(t,\cdot)}{\alpha/2,\alpha}
 \]
 for some $\theta \in (0,1/2)$ depending on $\alpha$. After integration we obtain
\begin{equation}\label{eq_mfg_parabolic_smooth_MN}
  \int_0^T \norm{Du(t,\cdot)}{L^q(\Omega)}^q \, dt \leq C_1 \norm{u}{\alpha/2,\alpha}^{q(1-\theta)} \int_0^T \norm{D^2 u(t,\cdot)}{L^{q/2}(\Omega)}^{q\theta} \, dt + C_2T \norm{u}{\alpha/2,\alpha}^q.
 \end{equation}
 Combining \eqref{eq_mfg_parabolic_smooth_HJ_Lp_estimate} and \eqref{eq_mfg_parabolic_smooth_MN}, we deduce that
 \[
  \norm{u}{W^{1,2}_{q/2}} \leq C \left(1 + C_1^{1/q} K^{1 - \theta} \norm{u}{W^{1,2}_{q/2}}^{2 \theta} + (C_2 T)^{1/q} K + \norm{f}{L^{q/2}(Q_T)} + \norm{u_0}{W^{2,q/2}} \right)
 \]
 and we obtain the required estimate by applying Young's inequality, since $2 \theta < 1$.

 \subsection{Proof of \cref{thm_mfg_wp}}
 \label{section:proof_mfg_wp}

 Define the linear operators $S_I$, $S_T$, $S_{IS}$ and $S_{TS}$ according to \eqref{eq_mfg_parabolic_smooth_SI}, \eqref{eq_mfg_parabolic_smooth_ST}, \eqref{eq_mfg_parabolic_smooth_SIS} and \eqref{eq_mfg_parabolic_smooth_STS}, respectively. From \cref{prop_parabolic_FP_wellposed,prop_parabolic_FP_DGNM,prop_parabolic_W2p}, we have that
\[
 S_I \in \cL(C^{\alpha}(\Omega),\cH^1_2(Q_T) \cap L^\infty(Q_T)), \quad S_{IS} \in \cL(L^{q/2}(0,T; W^{-1,q/2}(\Omega)),\cH^1_{2}(Q_T)),
\]
where $\alpha$ is fixed in \ref{h_parabolic_mfg}, and
\[
 S_T \in \cL(W^{2,\infty}, W^{1,2}_{q/2}(Q_T) \cap L^\infty(Q_T)), \quad S_{TS} \in \cL(L^{q/2}(Q_T), W^{1,2}_{q/2}(Q_T) \cap L^\infty(Q_T)).
\]
Using \cref{prop_parabolic_FP_DGNM}, we observe that $\cS_1$ and $\cS_2$, defined in  \eqref{eq_mfg_parabolic_smooth_def_S}, have range contained in
\[
 Z = \left(W^{1,2}_{q/2}(Q_T) \cap L^\infty(Q_T) \right) \times \left (\cH^1_2(Q_T)\cap C^{\beta/2, \beta}(Q_T) \right )
\]
for some $\beta \in (0,1)$.
Setting
\[
  X =  W^{0,1}_q(Q_T) \times L^\infty(Q_T),
\]
and using \ref{h_parabolic_mfg}, we have that the mapping
\[
 \cS \cR \colon \left \{ \begin{array}{l} X \to Z \\ (v,\rho) \mapsto \cS_1 \cR_1(v,\rho) + \cS_2 \cR_2(v,\rho)\end{array} \right.,
\]
where $\cR_1$ and $\cR_2$ are defined in \eqref{eq_mfg_parabolic_smooth_def_R}, is well-defined, continuous and maps bounded subsets of $X$ to bounded subsets of $Z$. Moreover, as a consequence of \cref{lem:parabolic_compact_embedding} and the Arzela-Ascoli theorem, we know that the embedding $Z \hookrightarrow X$ is continuous and compact. In particular, it follows that $\cS \cR \colon X \to X$ is continuous and compact.
We are therefore going to apply the Leray-Schauder fixed point theorem \cite[Theorem 11.3]{GT2001} to $\cS \cR$.
We have to prove that the set
\[
 \Sigma := \left \{ (u_\sigma,m_\sigma) \in X : \, (u_\sigma, m_\sigma) = \sigma \cS \cR(u_\sigma, m_\sigma), \, \sigma \in [0,1] \right \}
\]
is bounded. Notice that if $(u_\sigma, m_\sigma) \in \Sigma$, then
\begin{equation}
 \begin{cases}
    - \partial_t u_\sigma - \Delta u_\sigma + \sigma H(x,Du_\sigma) = \sigma F[m_\sigma] \quad & \textnormal{in } (0,T) \times \Omega, \\
    \partial_t m_\sigma - \Delta m_\sigma - \diver \left ( \sigma m_\sigma H_p(x,Du_\sigma) \right )= 0 \quad & \textnormal{in } (0,T) \times \Omega, \\
    u_\sigma = m_\sigma = 0 \quad & \textnormal{on } (0,T) \times \partial \Omega \\
    u_\sigma(T,\cdot) = \sigma u_T, \quad m_\sigma(0,\cdot) = \sigma m_0 \quad & \textnormal{in } \Omega.
 \end{cases}
\end{equation}
We deduce from \cref{prop_parabolic_FP_wellposed,prop_parabolic_FP_DGNM,thm_mfg_parabolic_HJ_wp} that there exists $C > 0$, independent of $\sigma \in [0,1]$, such that
\[
 \norm{(u_\sigma, m_\sigma)}{X} \leq C.
\]
We can apply the Leray-Schauder theorem, which proves existence of a weak solution to \eqref{eq_mfg_parabolic_smooth}.

  We now turn to the proof of uniqueness, which follows the lines of the one in \cite{LL2006}. Let
  \[
   (u_1,m_1), \, (u_2,m_2) \in \left (\cH^1_2(Q_T) \cap W^{0,1}_{r}(Q_T) \right) \times \left (\cH^1_2(Q_T) \cap L^{2(d+2)/d}(Q_T)\right)
  \]
  be two weak solutions to \eqref{eq_mfg_parabolic_smooth}, where $r = \max\{4,d+2 \}$. Set $v = u_1 - u_2$ and $\rho = m_1 - m_2$. Using $\rho$ as a test-function for the equation satisfied by $v$ and conversely, we obtain
\begin{multline*}
 \int_0^T \left \langle - \partial_t v(t), \rho(t) \right \rangle_{H^{-1},H^1_0} + \int_{\Omega} D v(t,x) \cdot D\rho(t,x)\, dx dt \\ + \int_0^T \int_\Omega \left(H(x,Du_1(t,x)) - H(x,Du_2(t,x)) \right) \rho(t,x)\, dx dt \\
 = \int_0^T \int_{\Omega} \left(F[m_1(t)](x) - F[m_2(t)](x)\right) (m_1(t,x) - m_2(t,x))\, dx dt
\end{multline*}
and
\begin{multline*}
 \int_0^T \left \langle \partial_t \rho(t), v(t) \right \rangle_{H^{-1},H^1_0} + \int_{\Omega} D\rho(t,x) \cdot Dv(t,x)\, dx dt \\ = - \int_0^T \int_{\Omega} \left(m_1(t,x) H_p(x,Du_1(t,x)) - m_2(t,x) H_p(x,Du_2(t,x)) \right) \cdot Dv(t,x)\, dx dt
\end{multline*}
Subtracting the second equality from the first and using the fact that
\begin{align*}
  \int_{0}^T \left \langle \partial_t v(t), \rho(t) \right \rangle_{H^{-1},H^1_0} & + \left \langle \partial_t \rho(t), v(t) \right \rangle_{H^{-1},H^1_0} \, dt = \int_0^T \frac{\mathrm{d}}{\mathrm{d}t} \int_{\Omega}\rho(t,x) v(t,x) \, dx dt \\
  & \quad = \int_{\Omega} \rho(T,x) v(T,x) - \rho(0,x) v(0,x)\, dx = 0,
\end{align*}
we obtain
\begin{align*}
 & \int_0^T \int_{\Omega} \left(F[m_1(t)](x) - F[m_2(t)](x)\right) (m_1(t,x) - m_2(t,x))\, dx dt \\
 & \qquad = \int_0^T \int_\Omega m_1 \left(H(x,Du_1) - H(x,Du_2) + H_p(x,Du_1)\cdot(Du_2 - Du_1)) \right) dx dt \\
 & \qquad \quad + \int_0^T \int_\Omega m_2 \left(H(x,Du_2) - H(x,Du_1) + H_p(x,Du_2)\cdot(Du_1 - Du_2)) \right) dx dt.
\end{align*}
Using the non-negativity of $m_1$ and $m_2$ and the Lasry-Lions condition, we deduce $m_1 = m_2$. We may then apply the comparison principle \cref{prop_comparison_principle} to conclude that $u_1 = u_2$.

\bibliographystyle{abbrv}
\bibliography{bibliography}
\end{document}